\documentclass{amsart}
\usepackage{graphicx} 
\usepackage{amsmath,amssymb,amsthm}
\usepackage{amsbsy}
\usepackage[utf8]{inputenc}
\usepackage{mathrsfs}
\usepackage{version, tabularx, multicol}
\usepackage{graphicx,float,psfrag}
\usepackage{subfigure}
\usepackage{epstopdf}
\usepackage{color,latexsym,amsfonts}
\usepackage{mathtools}
\usepackage[pagewise]{lineno}
\usepackage{pifont} 
\usepackage[shortlabels,inline]{enumitem}
\usepackage{stmaryrd}
\usepackage{pifont} 
\usepackage{hyperref}
\usepackage{tikz}

\newcommand{\cuthere}{%
\noindent
\raisebox{-2.8pt}[0pt][0.95\baselineskip]{\ding{34}}
\unskip{\tiny\dotfill}
}
\newcolumntype{M}[1]{>{\centering\arraybackslash}m{#1}}

\newcommand{\bN}{{\mathbf N}}
\newcommand{\bH}{{\mathbf H}}

\newcommand{\cc}{{\mathcal C}}

\newcommand{\cf}{{\mathcal F}}

\newcommand{\ci}{{\mathcal I}}
\newcommand{\cj}{{\mathcal J}}
\newcommand{\ck}{{\mathcal K}}

\newcommand{\cl}{\mathcal L}

\newcommand{\cn}{{\mathcal N}}
\newcommand{\cm}{{\mathcal M}}

\newcommand{\height}{\mathbf{\mathrm{height}}}
\newcommand{\node}{\mathbf{\mathrm{node}}}
\newcommand{\mass}{\mathbf{\mathrm{mass}}}

\newcommand{\ct}{{\mathcal T}}

\newcommand{\cth}{\ct^{\height}_h}
\newcommand{\wck}{\widehat {\ck}}
\newcommand{\ckh}{\ck^{\height}_h}
\newcommand{\ckH}{\ck^{\height}_H}
\newcommand{\ckhp}{\ck^{\height}_{h'}}
\newcommand{\ctn}{\ct^{\node}_\delta}

\newcommand{\ctnr}{\ct^{\node,*}_\delta}
\newcommand{\ckn}{\ck^{\node}_\delta}
\newcommand{\cknr}{\ck^{\node,*}_\delta}
\newcommand{\cknrp}{\ck^{\node,*}_{\delta'}}
\newcommand{\cknD}{\ck^{\node}_\Delta}
\newcommand{\ctm}{{\mathcal T}^{\mass}_r}
\newcommand{\ctmun}{{\mathcal T}^{\mass}_1}
\newcommand{\ckm}{\ck^{\mass}_r}
\newcommand{\ckmun}{\ck^{\mass}_1}

\newcommand{\br}{{\mathrm R}}
\newcommand{\wbr}{\widehat {\mathrm R}}

\newcommand{\E}{{\mathbb E}}

\newcommand{\N}{{\mathbb N}}

\renewcommand{\P}{{\mathbb P}}

\newcommand{\R}{{\mathbb R}}

\newcommand{\T}{{\mathbb T}}

\newcommand{\rd}{{\rm d}}

\newcommand{\bm}{\mathbf m}
\newcommand{\bt}{{\mathbf t}}

\newcommand{\bS}{{\mathbf S}}
\newcommand{\bT}{{\mathbf T}}
\newcommand{\bV}{{\mathbf V}}
\newcommand{\bZ}{{\mathbf Z}}

\newcommand{\bY}{{\mathbf Y}}

\newcommand{\bs}{{\mathbf s}}

\newcommand{\ind}{{\bf 1}}

\renewcommand{\root}{{\varrho}}

\newcommand{\Br}{{\mathrm {Br}}}

\newcommand{\dlgh}{{d_\mathrm {LGH}}}

\newcommand{\dlghu}{{d^{(1)}_\mathrm{LGH}}}

\newcommand{\supp}{{\mathrm {supp}}}

\newcommand{\inv}[1]{\mathop{\frac{1}{ #1}}\nolimits}
\newcommand{\expp}[1]{\mathop {\mathrm{e}^{ #1}}}

\newcommand{\II}[1]{\llbracket #1 \rrbracket}

\newcommand{\pot}{{\mathrm {U}}}
\newcommand{\psicr}{\psi_{\mathrm {crit}}}

\newcommand{\Lf}{{\rm Lf}}
\newcommand{\Sk}{{\rm Sk}}
\newcommand{\spine}{\llbracket 0, \infty  \rrparenthesis}

\newcommand{\lb}{[\![}
\newcommand{\rb}{]\!]}

\theoremstyle{plain}
\newtheorem{theo}{Theorem}[section]
\newtheorem*{theo*}{Theorem}

\newtheorem{cor}[theo]{Corollary}
\newtheorem*{cor*}{Corollary}
\newtheorem{prop}[theo]{Proposition}
\newtheorem{lem}[theo]{Lemma}
\newtheorem{defi}[theo]{Definition}
\theoremstyle{remark}
\newtheorem{rem}[theo]{Remark}

\author{Romain Abraham}
\address{
Institut Denis Poisson
Universit\'e d'Orl\'eans, Universit\'e de Tours, CNRS,
B.P. 6759,
45067 Orl\'eans cedex 2,
France.}
\author{Jean-Fran{\c c}ois Delmas}
\address{CERMICS, ENPC, Institut Polytechnique de Paris, CNRS,  Marne La Vall\'ee, France.}

\title{Coupling some conditioned  L\'evy trees with the Kesten tree}
\date{\today}

\begin{document}

\begin{abstract}
    We consider locally compact  Lévy trees conditioned to be large, with respect to different criterion: its height, its maximal ``size'' vertex and its total ``mass''. In the critical case, we provide a coupling with a truncated Kesten tree which then allows to directly prove the local convergence in distribution of the conditioned Lévy tree to be large towards the Kesten tree. We also consider the sub-critical and  super-critical cases. In the former case the results can be partial, due to a possible condensation phenomenon which is outside the mathematical framework  used in this paper. 
\end{abstract}


\maketitle

\section{Introduction}

Local limits (in distribution) of Bienaym\'e-Galton-Watson (BGW) trees conditioned to be large have been extensively studied in the recent years, and various conditioning have been considered: on the height \cite{Kesten1986}, on the total population size  \cite{Aldous1998a}, on the total number of leaves \cite{Kortchemski2012}, see also \cite{Abraham2026} and references therein for other examples. When the BGW tree is critical (the mean value of the offspring distribution is 1), all these conditionings converge locally  to the same  limit: the size-biased tree also called Kesten tree. 
(Of course, not all conditioning of critical 
BGW tree converges locally to the Kesten tree, see~\cite{abd:fat}.) 

Scaling limits of BGW trees are given by random real trees called (continuum) Lévy trees, see~\cite{Duquesne2002}, and here denoted by $\ct$; and analogous results also hold for conditioned critical  L\'evy trees \cite{Lambert2007, Duquesne2009a}  and the limiting tree is the continuum analogue of the discrete Kesten tree. For convenience, we will still call this random real tree a Kesten tree in this paper and denote it by $\ck$; it is composed of one semi-infinite branch called the spine on which are grafted random real trees (whose distribution is closely related to the initial L\'evy tree, see Section~\ref{sec:kesten}). This tree appears first in the Brownian tree context in \cite{Aldous1991} and is called the the self-similar continuum random tree, and is constructed for a general branching mechanism $\psi$  using exploration processes in \cite{Duquesne2009a}.
Let us mention that the Lévy tree are given under the so-called excursion measure $\N^\psi$, and that the Kesten tree is a random variable defined under the probability measure $\P^\psi$, see Sections~\ref{sec:LT}-\ref{sec:kesten}. 

\medskip

The goal of this paper is to give a coupling between the L\'evy trees conditioned to be large and a truncation/transformation of the Kesten tree, so that we can read the convergence in distribution of 
the critical L\'evy trees conditioned to be large towards the Kesten tree as an a.s.\ (local) convergence of the truncated Kesten tree to the (untruncated) Kesten tree. 
We shall consider three kind of conditioning for  the Lévy tree $\ct$ to be large, that is, to have:
\begin{enumerate}[(i)]
    \item\label{it:h}  its height $\bH(\ct)$ equal to  $h$,
    \item\label{it:n}   a maximal vertex with ``size'' $\Delta(\ct)$ equal to   $\delta$,
    \item\label{it:m}  its  total ``mass'' $\sigma=\sigma(\ct)$ equal to $r$, 
\end{enumerate}
and then let this parameter goes to infinity. 

\medskip
A Lévy tree (resp. Kesten tree) represents the genealogy of a continuous state branching (CB) process (resp. CB with immigration process), so that the previous coupling could also have been stated for conditioned CB process and trunctated CB process with immigration; however the picture is much more intuitive when using the genealogical structure given by the tree. 
To state precisely the results, we consider Lévy trees as random variables taking values in the Polish metric space $(\T, \dlgh)$ of (Gromov-Hausdorff isometric classes of) rooted complete locally compact real trees. To do so requires some usual hypothesis on the branching mechanism $\psi$ associated with the Lévy tree or the corresponding CB process such as conservativeness (that is, non explosion in finite time) and Grey condition (that is, local compactness of the Lévy trees); we also assume the CB process has infinite variation and thus the total length of the Lévy tree is infinite 
(see Section~3 in~\cite{LeGall1998b} when this last hypothesis is not fulfilled). 

Some of the results for the critical case ($\psi'(0)=0$) have extension to the  sub-critical ($\psi'(0)>0$)
or super-critical ($\psi'(0)<0$) cases. 
Let us mention that a super-critical Lévy tree  $\ct$ conditioned to be finite is distributed as a sub-critical Lévy tree, 
see the Girsanov transformation~\eqref{eq:Girsanov-sigma} with $\psi$ the branching mechanism associated to $\ct$ and 
$\theta$ replaced by $\theta_0$ the largest root of $\psi=0$. 
Let us mention that there are other possible limits than the Kesten tree for local limits (in distribution) of      Lévy trees conditioned  to be large, see~\cite{adhe-2022}.

\subsection{Lévy trees conditioned to have large height}
Let us first consider the tree conditioned on having height $h$. To construct this tree from the Kesten tree $\ck$, we cut the spine of $\ck$ at level $h$ (and forget everything that is above this cut-point), and also remove all the sub-trees grafted on the remaining spine that reach level $h$. This gives a tree $\ckh$, whose construction is illustrated in Fig.~\ref{fig:height}. We denote by $h$ the vertex on the spine at distance $h$ from the root $\root$, and see it as a distinguished vertex of $\ckh$. 

\begin{figure}[ht]
 \includegraphics{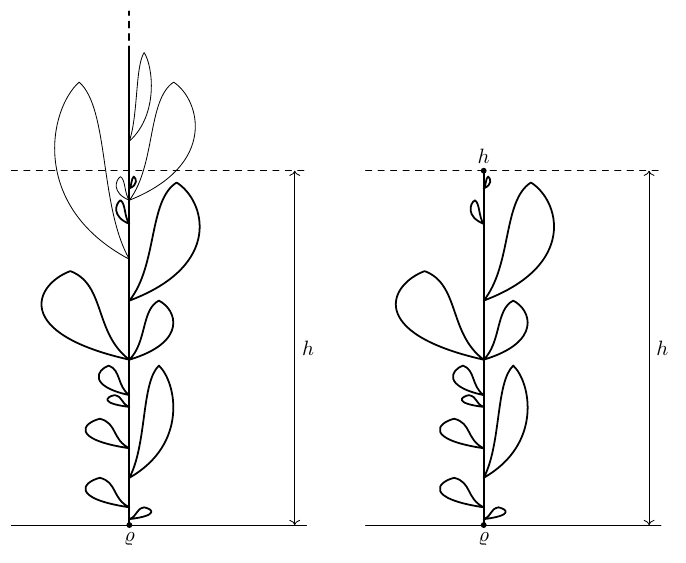}
     \caption{Left: the Kesten tree (a spine given by an semi-infinite branch on which are grafted Lévy sub-trees); Right: the truncated Kesten tree $\ck_h^{\rm height}$ with the distinguished vertex $h$ on the spine being the only vertex at distance $h$ from the root $\root$.}
    \label{fig:height}
\end{figure}

We summarize Theorem~\ref{theo:height}
on the coupling, and the local convergence of the truncated Kesten trees from Proposition~\ref{prop:height} and the direct consequence of the local convergence in distribution of the conditioned Lévy tree to have height $h$, say $\cth$,  towards the Kesten tree as $h$ goes to infitny, see Theorem~\ref{theo:cv-loc-height}. 

 \begin{theo*}[Conditioning w.r.t.\ the height]
We assume the infinite variation setting and the Grey condition. 
The (sub-)critical Lévy tree $\ct$ conditioned under the excursion measure to have height $h$, say $\cth$, with its only vertex $X_h$ at distance $h$ from the root $\root$, is distributed as $(\ckh,h)$:
\[
(\cth, X_h)
   \, \overset{\text{(d)}}{=} \,
   (\ckh,h).
\]
Furthermore 
 the ``increasing'' sequence $(\ckh)_{h\ge 0}$  converges  to the Kesten tree $\ck$ in $(\T, \dlgh)$. In particular, we have the following  local   convergence in distribution:
 \begin{equation}
   \label{eq:limTh-intro-height}
      \cth  \xrightarrow[h\rightarrow
   +\infty]{\text{(d)}}
      \ck
     \quad\text{in}\quad
  (\T, \dlgh).
 \end{equation}
 \end{theo*}

 The convergence~\eqref{eq:limTh-intro-height} was already proven in~\cite{Abraham2009a}, and the coupling was implicit. 
We refer to Remark~\ref{rem:super-crit-H} for the super-critical case. 

\subsection{Lévy trees conditioned to have large maximal vertex ``size''}
According to~\cite{Duquesne2002},
the   L\'evy  tree $\ct$  can  be
constructed using  a coding by the  so-called height process which  is a
functional  of a  spectrally positive  L\'evy process  $\bV$ with  Laplace
exponent $\psi$. Then, each 
jump of  the process $\bV$ corresponds to the ``size'' of a vertex of infinite degree in $\ct$. 
We refer to~\eqref{DefMas} in Section~\ref{sec:levy-tree}~(vi)
for a more
intrinsic way to define the ``size'' of a vertex as the limit, when $\varepsilon$ goes to $0$,  of the number of sub-trees
attached to this point with  height larger than  $\varepsilon$ with a correct renormalization. 

\medskip 

The  local
limit  in distribution of the  L\'evy tree conditioned  of having  maximal vertex  ``size''
$\delta$, say $\ctn$,   is  obtained  in  \cite{nassif2022} in  the  critical  and  sub-critical
cases under further assumption on the Lévy measure $\pi$ of the branching mechanism. To simplify, in the introduction, we shall consider the critical case and assume that $\pi$ has no atom   and support $(0, +\infty)$ in $(0, +\infty)$. 
Then the tree $\ctn$ has exactly one node of ``size'' $\delta$, say $X_\delta$, and we denote by $\ctnr$  the tree  $\ctn$ when one remove the sub-tree above $X_\delta$. 
We now provide  a representation of $\ctnr$
as a truncation of the Kesten tree. 
Let $H_\delta$
be  the vertex on the spine of the Kesten tree $\ck$ 
being the lowest  branching point of the spine on  which is grafted a
tree with a vertex of ``size''  $\ge \delta$ (which is possibly its root which belongs to the spine). Then we cut the spine of $\ck$ at  $H_\delta$,
and remove the sub-tree grafted at $H_\delta$ to get the tree $\cknr$. Then, we built the tree $\ckn$ by grafting at $H_\delta$ in $\cknr$ a sub-tree whose root is a branching point with ``size'' $\delta$ and which is distributed as the Lévy tree conditioned to have no other vertices of ``size'' $\ge \delta$. 
See Fig.~\ref{fig:node} for a representation of $ \ck$, $\ckn$ and $\cknr$. 

\medskip

\begin{figure}[ht]
    \centering
    \includegraphics{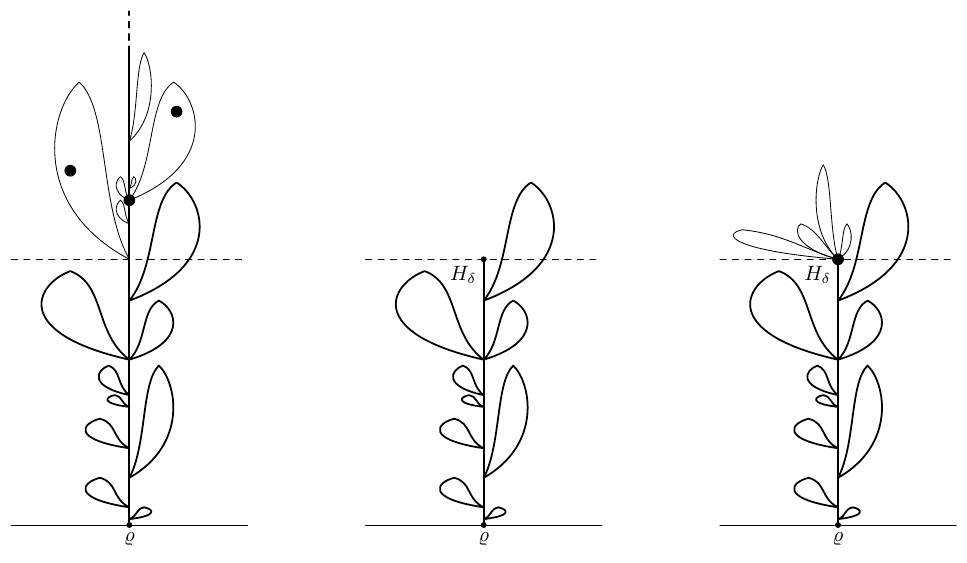}
     \caption{Left: the Kesten tree with vertices of ``size'' $\ge \delta$; Center: the truncated Kesten tree $\cknr$ with the distinguished vertex $H
     _\delta$ on the spine;
    Right: the tree $\ckn$, with an independent sub-tree grafted at $H_\delta$ whose root is of ``size'' $=\delta$  and which does not have any vertex of ``size'' $\ge \delta$.}
    \label{fig:node}
\end{figure}

We summarize Theorem~\ref{theo:Delta}
on the coupling, and the local convergence of the truncated Kesten trees from Proposition~\ref{prop:delta} and the direct consequence of the local convergence in distribution of the conditioned Lévy tree to have maximal vertex ``size'' equal to $\delta$, say $\ctn$,  towards the Kesten tree as $\delta$ goes to infitny, see Theorem~\ref{theo:cv-loc-size-vertex}.

 \begin{theo*}[Conditioning w.r.t.\ the maximal vertex ``size'']
We assume the infinite variation setting and the Grey condition, as well as the Lévy measure $\pi$ has no atom and full support in $(0, + \infty)$. 
The (sub-)critical Lévy tree $\ct$ conditioned under the excursion measure to have maximal vertex ``size'' $\delta$, say $\ctn$, with its only vertex $X_\delta$ of ``size'' $\delta$
is distributed as $(\ckn,H_\delta)$:
\begin{equation}
    \label{eq:coupling-intro-node}
(\ctn, X_\delta)
   \, \overset{\text{(d)}}{=} \,
   (\ckn, H_\delta)
   \quad\text{and}\quad
   \ctnr
   \, \overset{\text{(d)}}{=} \,
\cknr.
\end{equation}
Furthermore in the critical case, 
 the ``increasing'' sequence $(\cknr)_{h\ge 0}$  converges  to the Kesten tree $\ck$ in $(\T, \dlgh)$, and we have the following  local   convergence in distribution:
 \begin{equation}
   \label{eq:limTh-intro-node}
      \ctn  \xrightarrow[\delta \rightarrow
   +\infty]{\text{(d)}}
      \ck
     \quad\text{in}\quad
  (\T, \dlgh).
 \end{equation}
 \end{theo*}

 The convergence~\eqref{eq:limTh-intro-node} was already proven in~\cite{nassif2022} and the sub-critical case is also considered therein. The coupling~\eqref{eq:coupling-intro-node}
 is also valid in the sub-critical case with a slightly different definition of the cutting vertex $H_\delta$, see Theorem~\ref{theo:Delta}; and in this case the local limit, when $\delta$ goes to infinity, of $\ctnr$ is the truncated Kesten tree $\wck_{E}$ whose spine is cut at vertex $E$ at distance from the root $\root$ given by an independent exponential random variable with mean $1/\psi'(0)$. 
Intuitively, the limit, when $\delta$ goes to infinity, of $\ctn$ 
in this case, would consist of the compact tree $\wck_{E}$ on which is grafted at $E$ a (non locally compact) 
Lévy sub-tree whose root is a branching point with infinite ``size'' (which implies the non local compactness). This is the so-called condensation phenomenon.
The super-critical case can be deduce from the sub-critical case by the Girsanov transformation~\eqref{eq:Girsanov-sigma} (with $\theta$ the positive root of $\psi=0$), as on the non-extinction event the maximal vertex ``size'' is $+\infty$ as the support of the Lévy measure $\pi$ is unbounded, and thus conditioning on the maximal vertex ``size'' to be  $\delta$ implies the extinction of the super-critical Lévy tree. 

\subsection{Lévy trees conditioned to have large total mass}
 Eventually, let us recall that a Lévy tree is naturally endowed with a mass measure supported by its leaves, see Section~\ref{sec:levy-tree}~(iv). Let us denote by $\sigma$ the total mass of this measure, which can also be seen as the life-time of the excursion of the height process that codes the Lévy tree. 

To simplify, we consider the critical case ($\psi'(0)=0$). In this case, the coupling is less straightforward. Consider the Kesten tree $\ck$, and denote by $(h_i, \sigma_i, \ct_i)_{i\in I}$ the height $h_i$ at which the sub-tree $\ct_i$ is grafted on the spine, and $\sigma_i$ denote its mass, see Fig.\ref{fig:mass1}. The process $\bS^\ck$ defined $S^\ck_h=\sum_{i\in I}  \sigma_i \ind_{\{ h_i \leq h\}}$ is then a subordinator started at $0$ with Laplace exponent $\varphi=\psi'\circ \psi^{-1}$, see Section~\ref{sec:Levy-cond-mass}. For simplicity, we write $\bS$ for $\bS^\ck$. 

\begin{figure}[ht]
    \centering
    \includegraphics{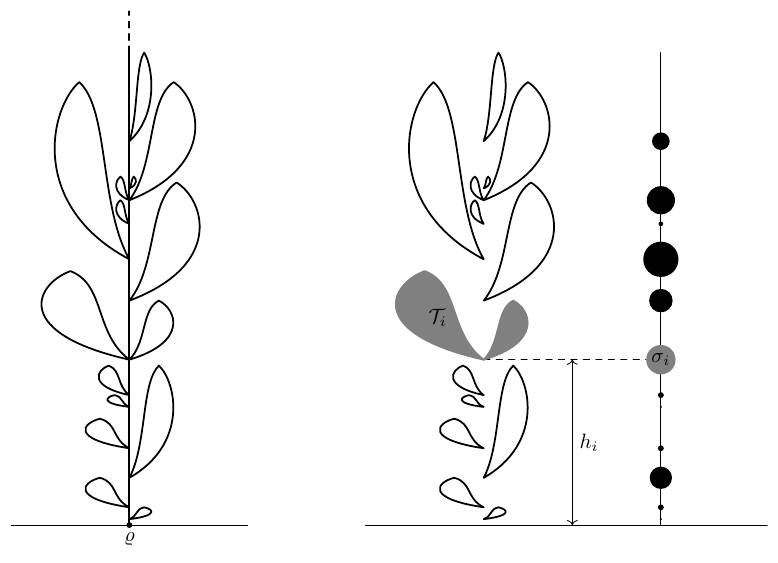}
     \caption{Left: the Kesten tree;  Center: the sub-trees grafted on the spine, with in gray an instance of a tree $\ct_i$ (with the root being a branching point) grafted at height $h_i$ with mass $\sigma_i$; Right: mass of the grafted trees.}
    \label{fig:mass1}
\end{figure}

The idea, is to consider the subordinator $\bS$ started at $0$ and conditioned to die continuously at level $r$, which we denote by $\bS^r$ and write $\zeta^r$ for its lifetime (in particular $S_{\zeta^r-}^r=r$ a.s.). Following~\cite{krs:csemp}, provided the potential of $\bS$ has a continuous density, say $u$, w.r.t.\ the Lebesgue measure on $(0, +\infty)$, that is, $\int_{\R^+} \P(S_t \le x) \, \rd t=\int _{(0, x]} u(y)\, \rd y$ for all $x>0$ and $u$ is continuous, then the process $\bS^r$ is well defined and with  $\P^\varphi$ and $\P ^{\varphi,r}$ denoting the distribution of $\bS$ and $\bS^r$, we get for all $h\ge 0$ and $F$ a non-neagtive measurable functional: 
\[
  \E^{\varphi,r}\left[ \ind_{\{h<\zeta^r\}}\, F(S^r_{[0, h]})\right]= \E^\varphi\left[
\frac{u(r-S_h)}{u(r)} \ind_{\{ S_h <  r\}}\,  F(S_{[0, h]})\right].
\]
Let us mention that $u(r)/r$ is also the density w.r.t.\ the Lebesgue measure of the total mass of the Lévy tree under the excursion measure $\N^\psi$, see~\eqref{eq:r*pi=u}.

From the process $\bS^r$, we build a modified truncated Kesten tree, say $\ckm$,  by considering a spine of length $\zeta_r$ on which we graft at height $h_j$ an independent  Lévy tree $\ct_j$ with total mass $\sigma_j$, where $(h_j, \sigma_j)_{j\in J}$ are the jumping times and jums of the process $\bS^r$. In particular the total mass of $\ckm$ is given by $\sum_{j\in J} \sigma_j=S_{\zeta_r-}^r=r$. 
(We stress the distribution of $\ct_j$ conditionally on $\sigma_j=r$ is the conditional distribution of a grafted tree $\bN^\psi[\rd \ct \, |\, \sigma=r]$ with the grafting measure $\bN$ defined in~\eqref{eq:def-bN}.)
The tree $\ckm$ is not a functional of the Kesten tree. However, in the stable case $\psi (\lambda)=\lambda^a $ with $a\in (1, 2]$, 
one could use a scaling argument to recover the jumps $(\sigma_j)_{j\in J}$
from the jumps $(\sigma_i)_{i\in I}$, see~\cite{cub:bridges-11}, and then use that conditionally on $\sigma_j$
the tree $\ct_j$ is distributed as a scaled normalized random tree $\ct'$ with total mass 1. We also recall that in the stable case one can get a regular version of $\ctm$
the Lévy tree conditioned to have total mass $r>0$ by scaling, see Section~3.3 in~\cite{Duquesne2002}.

\medskip

We get from Corollary~\ref{cor:mas} that 
$\ctm$, the Lévy tree conditioned to have total mass $r>0$, is distributed as $\ckm$ (in fact the distribution $\ckm$ provide a regular version of the distribution of $\ctm$), and from Proposition~\ref{prop:mass} (see also 
Theorem~\ref{theo:cv-loc-mass-crit}) we get  the local convergence in distribution of $\ctm$   towards the Kesten tree as $r$ goes to infinity when the potential density is non-increasing (this relies on the Fatou lemma and the fact that $\liminf_{r\rightarrow \infty} \frac{u(r-S_h)}{u(r)} \ind_{\{ S_h <  r\}} \geq 1$ which is a direct consequence of $u$ being non-increasing).

 \begin{theo*}[Conditioning w.r.t.\ the total mass]
We assume the infinite variation setting and the Grey condition, as well as the existence of continuous density on $(0, +\infty)$, say $f_\sigma$, of the total mass under the excursion measure $\N^psi$ such that $r\mapsto r f_\sigma(r)$ is non-increasing. 
The critical Lévy tree $\ct$ conditioned under the excursion measure to have total mass $r$, say $\ctm$, 
is distributed as $\ckm$:
\begin{equation}
    \label{eq:coupling-intro-mass}
\ctm
   \, \overset{\text{(d)}}{=} \,
   \ckm, 
\end{equation}
and 
 the sequence $(\ctm)_{r> 0}$  converges  to the Kesten tree $\ck$ in $(\T, \dlgh)$, and we have the following  local   convergence in distribution:
 \begin{equation}
   \label{eq:limTh-intro-mass}
      \ctm  \xrightarrow[r \rightarrow
   +\infty]{\text{(d)}}
      \ck
     \quad\text{in}\quad
  (\T, \dlgh).
 \end{equation}
 \end{theo*}

Using a Girsanov transformation, we have a similar result for the super-critical case, see Corollary~\ref{cor:cv-loc-mass-generic}. The sub-critical case is more delicate: in the generic case (where one can extend the functionn $\psi$ to negative values such that there exists a root $\theta^*$ of $\psi'=0$), then one can get a similar result, see Corollary~\ref{cor:cv-loc-mass-generic}. In the non-generic case, one expect to have a condensation phenomenon similar to what is observed in the discrete case, 
see~\cite{js11,Janson2012,Abraham2014a}.
But in the continuous setting the formalism is more delicate 
(as the conjectured limiting tree is no more locally compact)
and the proof of the condensation is still an open question.

\section{Notations}

\subsection{Real trees}
\label{sec:real-tree}
We refer to \cite{Evans, lg:rta, j:rt} for a general presentation of
random real  trees.  Informally,  real trees  are metric  spaces without
loops,  locally  isometric   to  the  real  line.    More  precisely,  a
(non-empty) metric  space $(\bt,d)$  is  a  real  tree if  the  following
properties are satisfied:
\begin{enumerate}
	\item For every $x,y\in \bt$, there is a unique isometric map $f_{x,y}$
from $[0,d(x,y)]$ to $\bt$ such that $f_{x,y}(0)=x$ and $f_{x,y}(d(x,y))=y$.
	\item For every $x,y\in \bt$, if $q$ is a continuous injective map from
$[0,1]$ to $\bt$ such that $q(0)=x$ and $q(1)=y$, then
$q([0,1])=f_{x,y}([0,d(x,y)])$.
\end{enumerate}
If $x,y\in  \bt$, we will note  $\II{x,y}$ the range of  the isometric map
$f_{x,y}$   described  above   and  $\llbracket   x,y  \llbracket$   for
$\llbracket x,y \rrbracket \backslash  \{y\}$.  The degree $\kappa_x(\bt)$ of
$x\in \bt$ is the number  of connected components of
$\bt\setminus\{x\}$, and we simply write $\kappa_x$ for $\kappa_x(\bt)$
when there is no ambiguity on the tree $\bt$. We
shall consider  the set of leaves  $ \Lf (\bt)=\{x\in \bt,\,  \kappa_x= 1\}$
(with the convention that $\Lf(\bt)=\bt$ if  $\bt$ is reduced to one element),
the set of  branching vertices  $\Br(\bt)=\{x\in \bt, \,  \kappa_x\geq 3\}$ and
the        set        of         infinite        branching        vertices
$\Br_\infty(\bt) = \{  x\in \bt,\, \kappa_x = \infty \}  $.  The skeleton of
$\bt$  is   the  set   of vertices   in  the   tree  that   aren't  leaves:
$\Sk(\bt)=\bt\backslash \Lf (\bt)$.  A real tree $\bt$ is discrete if the set of
leaves   and   branching   points     is   locally   finite,   that   is,
$\{y\in \Lf(\bt)\cup \Br(\bt);  d(x,y)\leq a\}$ is finite for  all $a\geq 0$
and $x\in \bt$.  If $\bt$ is separable, then the set of leaves $\Lf(\bt)$ is a
Borel subset of  $\bt$, and there exists a unique  measure $L ^\bt$ on $\bt$ (endowed
with the Borel $\sigma$-field), called the length measure, such that:
\[
L^{\bt}(\Lf (\bt)) = 0
\quad\text{and}\quad
 L^{\bt}(\llbracket x,y
 \rrbracket)=d(x,y)
 \quad\text{for all $ x,y\in \bt$;}
\]
and furthermore the length measure is $\sigma$-finite. 
If the  real tree $(\bt,d)$ is complete and locally compact,
then any bounded closed set is compact by the Hopf-Rinow therorem, and
thus the real tree $\bt$ is a Polish metric space and thus separable.

We say that  $(\bt,d,\root)$ is a rooted real tree  with root $\root\in \bt$
if $(\bt,d)$  is a real tree  and $\root$ is a  distinguished vertex.  Let
$(\bt,d,\root)$  be  a  rooted  real  tree.  We define its height by:
\[
  \bH(\bt)=\sup_{x\in \bt}
  d(\root,x).
  \]
For every $x\in \bt$, $\lb\root ,x\rb$ is
interpreted as the ancestral line of the vertex $x$ in the tree. We define
a partial order on $\bt$ by setting $x\preccurlyeq y$ ($x$ is an
ancestor of $y$) if $x\in\lb \root,y\rb$.
For a non empty subset  $\bs\subset \bt$, there exists a unique $z\in \bt$, called the Most Recent
Common Ancestor (MRCA) of $\bs$  such that $\lb
\root,z\rb= \bigcap_{x\in \bs} \lb\root,x\rb$.
We simply write 
$x\wedge y$ for the MRCA of $x$ and $y$ (that is, of  $\bs=\{x,y\}$).

We shall consider  the infinite branch, denoted $\spine$,  as the rooted
real  tree  $(\R_+, d,  \root=0)$,  where  $d$  is the  usual  Euclidean
distance.

We  say  $(\bt,d,  \root,  x)$,  is  a  pointed  rooted  real  trees  if
$(\bt,d, \root)$ is a rooted real tree and $x\in \bt$ is a distinguished
vertex.  We shall consider the particular  case of the tree reduced to a
branch with one extremity of the branch being the root and the other the
pointed element; more precisely, for $h\geq 0$, we denote by $\II{0, h}$
the pointed rooted real tree $([0,h],  d, \root=0, h)$, where $d$ is the
usual Euclidean distance.

For simplicity, we shall write $\bt$ for a rooted real tree, and denote
by  $d^\bt$ its  distance  and by  $\root^\bt$  its  root.

\subsection{Grafting procedure}
\label{sec:graft}
We  recall the grafting procedure where  we  add (graft)
  rooted  real  trees   on  an  existing   rooted  real
trees. More  precisely, let $\bt$ be  a  rooted real  tree and let
$(x_i)_{i\in  I}$ and $(\bt_i, \root^{\bt_i},d^{\bt_i})_{i\in  I}$ be  a 
family respectively of vertices of $\bt$ and of
  rooted  real tree. 
We  set  $T  =  \bt \sqcup  \left(
  \bigsqcup_{i\in  I} \bt_i\backslash\{\root^{\bt_i}\}  \right)  $ where
the symbol $\sqcup$  means that we consider the disjoint union of the
sets $(\bt_i)_{i\in I}$ and $\bt$.  We set $\root^{T}=\root^\bt$. The set $T$ is
endowed with the following metric $d^{T}$: if $s,t\in T$,
\begin{equation*}
d^{T} (s,t) =
\begin{cases}
d^\bt(s,t)\ & \text{if}\ s,t\in \bt, \\
d^\bt(s,x_i)+d^{\bt_i}(\root^{\bt_i},t)\ & \text{if}\
s\in \bt,\ t\in \bt_i\backslash\{\root^{\bt_i}\} , \\
d^{\bt_i}(s,t)\ & \text{if}\ s,t\in \bt_i\backslash\{\root^{\bt_i}\} ,\\
d^\bt(x_i,x_j)+d^{\bt_j}(\root^{\bt_j},s)+d^{\bt_i}
(\root^{\bt_i},t)\
& \text{if}\ i\neq j \ \text{and}\ s\in \bt_j\backslash\{\root^{\bt_j}\} ,\ t\in
\bt_i\backslash\{\root^{\bt_i}\} .
\end{cases}
\end{equation*}
It is clear that
$(T,d^{T},\root^{T})$ is a  rooted real tree, and we
denote it by $T$ for simplicity and 
 use the following notation for the grafted tree:
\begin{equation}
T=\bt \circledast_{i\in I}(\bt_i,x_i) .
\label{def:gref}
\end{equation}
This construction can be easily extended to the case where $\bt$ is a
pointed rooted real tree with distinguished vertex $x\in \bt$,
which is also seen as a distinguished vertex of $T$.

\subsection{Sub-trees above/below a given level}
\label{sec:below/above-level}
Let $(\bt, d, \root)$ be a rooted real tree. Let
$h>0$. We  define the restriction  map $\br_h(\bt)$ which is  the sub-tree
below level $h$:
\begin{equation}
   \label{eq:def-rh}
  \br_h(\bt)=\{x\in \bt, \, d(\root,x)\le h\},
\end{equation}
and still denote  by $\br_h(\bt)$ the corresponding rooted  real tree with
root $\root$ and distance given by the restriction of $d$ to $\br_h(\bt)$.
Let   $(\bt_i^{\circ})_{i\in   I}$   be  the   connected   components   of
$\bt\setminus   \br_h(\bt)$;  let us  denote   by  $\root_i$   the  MRCA  (in $\bt$) of
$\bt_i^{\circ}$  and set  $\bt_i=\bt_i^{\circ}\cup\{\root  _i\}$ which  we
consider as a rooted real tree with root  $\root_i$.  
By construction, we have:
\[
  \bt=\br_h(\bt) \circledast_{i\in  I}(\bt_i,\root_i) .
\]
We will  also consider the
point measure:
\begin{equation}
   \label{eq:def-cna}
\cn_h^\bt=\sum_{i\in I}\delta_{(\root_i,\bt_i)}.
\end{equation}
The support of the measure $\sum_{i\in I}\delta_{\root_i}$  is a subset
of:
\begin{equation}
   \label{eq:t(h)}
  \bt(h)=\{x\in \bt\, \colon\, d(\rho,x)=h\}.
\end{equation}

We shall also cut a tree by removing the sub-tree above one of its
vertices. For $x\in \bt$, we define:
\begin{equation}
   \label{eq:def-wbr}
  \wbr(\bt,x)=\{x\} \cup \{y\in \bt, \, x\not\in \llbracket \root ,y \rrbracket
  \}. 
\end{equation}
Notice that $\br_{d(\root, x)}(\bt)\subset  \wbr(\bt,x)$.
 As for
$\br_h(\bt)$, we still denote 
by $\wbr(\bt,x)$ the corresponding rooted  real tree with
root $\root$ and distance given by the restriction of $d$ to
$\wbr(\bt,x)$; according to the context, we might see $x\in \wbr(\bt,x)$
as a distinguished vertex.

\subsection{Polish spaces of  rooted real trees}
\label{sec:T-T1}
By the Hopf-Rinow
therorem, if  $(\bt, d)$ is a complete and locally compact
metric  real tree, then every closed bounded subset of $\bt$ is
compact.
According to~\cite{adhe-2022}, one can  define a Gromov-Hausdorff metric
$\dlgh$ on the space $\T$ of (GH-isometric classes of) rooted complete
 locally compact real trees.  Furthermore,  following~\cite[Theorem~2.9]{adh:ghptwms},
we get that the space $(\T, \dlgh)$ is a Polish metric space.
When there is no possible confusion,  we  shall  also simply write  $\bt$  for its equivalence class in $\T$.

 One can  define a  restriction map  $\br_h$ on  $\T$ which  is consistent
 with~\eqref{eq:def-rh};  then the  map $(h,  \bt) \mapsto  \br_h(\bt)$ is
 continuous       from       $\R_+\times        \T$       to       $\T$,
 see~\cite[Lemma~5.4]{adhe-2022}.   As a  trivial  application from  the
 definition  of  the  GH-distance  $\dlgh$  in~\cite{adhe-2022},  let  us
 mention that for $h\geq 0$ and $\bt, \bt'\in \T$:
\begin{equation}
   \label{eq:rh(t)}
\br_h(\bt)=\br_h(\bt') \, \Longrightarrow\,\dlgh(\bt, \bt')\leq \expp{-h}.
\end{equation}

Similarly, one can also define a Gromov-Hausdorff metric $\dlghu$ on the
space  $\T_1$  of  (GH-isometric  classes of)  pointed  rooted  complete
locally compact real trees so that it is a Polish metric space; we refer
again to~\cite[Section~5.3]{adhe-2022} for the precise definition of the
restriction map which is again continuous.  One can define a restriction
map $\wbr$  on $\T_1$  which is consistent  with~\eqref{eq:def-wbr} and,
following~\cite[Lemma~6.22]{adhe-2022},     get     that     the     map
$ (\bt,x) \mapsto \wbr(\bt,x)$ is measurable from $ \T_1$ to $\T$ (or to
$\T_1$ if one consider $x$ as a distinguished vertex of $ \wbr(\bt,x)$).
The grafting map $((\bt, x),  \bt') \mapsto T=\bt \circledast (\bt', x)$
is a continuous  map from $\T_1\times \T$  to $\T$ (or to  $\T_1$ if one
consider the pointed tree $(T,x)$), see~\cite[Lemma~5.13]{adhe-2022}.

 \medskip

 Grafting a countable family of trees is more delicate, and we shall not
 consider it in full generality.  Write $\{\root\}$ for the tree reduced
 to  its   root.   Based  on~\cite[Lemma~5.31]{adhe-2022},  we   get  in
 particular  that  if $\cm=\sum_{i\in  I}  \delta_{(h_i,  T_i)}$ is  a
 Poisson    point    measure   on    $\R_+    \times    \T$   with    intensity
 $\rd h  \, \nu(\rd T)$,  where $\nu$  is a $\sigma$-finite  measure on
 $\T$ such  that $\nu(\bH(T)=0)=0$  and $\nu(  \bH(T)>\varepsilon)$ is
 finite for all $\varepsilon>0$, then the random trees:
\[
 \tau= \spine \circledast_{i\in I} (T_i, h_i),
   \quad
  \II{0,h} \circledast_{i\in I_h} (T_i, h_i)
  \quad\text{and}\quad
 \{\root\}\circledast_{i\in I_h} (T_i, \root),
 \]
 where  $h\geq 0$  and $I_h=\{i\in  I\, \colon\,  h_i\leq h\}$,  are well
 defined      $\T$-valued      random     variables.       The      tree
 $ \II{0,h} \circledast_{i\in I_h} (T_i, h_i)$ with distinguished vertex
 $h$ is also a well defined $\T_1$-valued random variable.

 Consider a non-decreasing  sequence $(A_r)_{r\geq 0}$ of measurable subsets of
 $\R_+\times \T$ and the random tree:
 \[
   \tau_r= \spine \circledast_{i\in I_{A_r}} (T_i, h_i),
 \]
 with $I_{A_r} =\{i\in I\, \colon\,  (h_i, T_i)\in A_r\}$, which we call
 the restriction of the tree $\tau$ to $A_r$.

 In this setting, we shall rewrite the monotonicity of the restrictions as:
\begin{equation}
   \label{eq:conv-inclu}
   \tau_r\subset \tau_{r'}
   \quad\text{a.s.\ for $0\leq r<r'$.}
 \end{equation} 
 If $\nu(A^c)=0$ for  $A=\lim_{r\rightarrow \infty } A_r$,
 then we get from~\cite[Lemma~5.31]{adhe-2022}  that a.s.\ $\lim_{r\rightarrow\infty } \dlgh(\tau_r, \tau)=0$,
 which we shall write as:
\begin{equation}
   \label{eq:conv-lim}
  \lim_{r\rightarrow \infty } \tau_r \stackrel{\text{a.s.}}{=} \tau
  \quad\text{in}\quad
  (\T, \dlgh).
  \end{equation}

\subsection{Branching process}
\label{sec:levy-tree}

We
consider the  branching mechanism $\psi$:
\begin{equation}
   \label{eq:psi}
\psi(\lambda)=\alpha'\lambda+\beta\lambda^2
+\int_{(0,+\infty)}\left(\expp{-\lambda
    r}-1+\lambda r\ind_{\{r\leq 1\}}\right)\pi(\rd r)
\quad\text{for}\quad \lambda\geq 0
,
\end{equation}
where   $\alpha'$ and  the quadratic parameter  $\beta$ are real numbers, the
Lévy measure 
$\pi$ on $(0, +\infty )$ is Borel, and:
\begin{equation}
   \label{eq:hyp-param}\tag{H0}
   \boxed{   \alpha'\in \R, \quad \beta\ge 0
     \quad\text{and}\quad
 \int_{(0,+\infty)}(1\wedge
r^2)\, \pi(\rd r)<+\infty.}
\end{equation}
We shall consider for $r\geq 0$:
\begin{equation}
  \label{eq:def-pid}
\bar \pi (r)=\pi((r, +\infty )). 
\end{equation}

The corresponding L\'evy process 
$\bV=(V_t)_{t\geq 0}$ is a process on $\R$ with independent and stationary
increments characterized by $V_0=0$ a.s.\ and:
\[
  \E^\psi\left[\expp{-\lambda V_t}\right]= \expp{ -
    t\psi(\lambda) }
  \quad\text{for all $\lambda\geq 0$ and $  t\geq 0$}.
\]
The corresponding  continuous state branching (CB) process $\bY=(Y_t)_{t\geq
  0}$ is an homogeneous Markov process on $[0, +\infty ]$ such that, for
$x\in \R_+$, $\P^\psi_x$-a.s.\ $Y_0=x$ and:
\begin{equation}
   \label{eq:Lap-Y}
   \E^\psi_x\left[\expp{-\lambda Y_t}\right]= \expp{ - x
     v(\lambda, t)}
   \quad\text{for all $\lambda> 0$  and $  t\geq 0$},
\end{equation}
where, for $\lambda> 0$, the function $t\mapsto v(\lambda, t)$ is the
unique positive solution of the equation:
\begin{equation}
   \label{eq:Lap-Y:v}
v(\lambda, t) + \int_0^t \psi(v(\lambda,s))\, \rd s = \lambda. 
\end{equation}

The  branching mechanism  $\psi$ is  convex, of  class $\cc^\infty  $ at
least on $(0, \infty )$, and $\psi(0)=0$. We simply write $\psi'(0)$ for
$\lim_{\lambda  \rightarrow   0+}  \psi'(\lambda)=\alpha'   -  \int_{(1,
  +\infty )} r\pi(\rd r)\in [-\infty ,  +\infty )$.  The CB process $\bY$
  and the branching mechanism  $\psi$ are critical (resp.\ sub-critical,
  resp.\  super-critical) if  $\psi'(0)=0$  (resp. $\psi'(0)>0$,  resp.\
  $\psi'(0)<0$).  We shall mainly consider the  (sub-)critical case; in
  this case  the integral  $\int_{(0,+\infty)}(r\wedge r^2)\pi(dr)$ is
  finite. When this latter integral is finite,  we can
  rewrite~\eqref{eq:psi} as follows and then define $(\alpha, \beta, \pi)$ as
   characteristic of the branching mechanism $\psi$:
 \begin{equation}
   \label{eq:psi-sub}
\psi(\lambda)=\alpha\lambda+\beta\lambda^2
+\int_{(0,+\infty)}\left(\expp{-\lambda
    r}-1+\lambda r\right)\pi(dr)
\quad\text{with}\quad \alpha= \psi'(0)\in \R.
\end{equation}

  \medskip
  
  We assume~\eqref{eq:hyp-param} and consider also the following assumptions.
  \begin{itemize}
\item Infinite variation:
   \begin{equation}
    \label{eq:variation} \tag{H1}
\boxed{    \beta>0
    \quad\text{or}\quad
    \int_{(0,1)}  r \pi(\rd r)=+\infty .}
\end{equation}
In particular, we then  have $\psi \neq 0$.  Under~\eqref{eq:hyp-param},
this   condition   is   equivalent    to the CB process   $\bY$   being   of   infinite
variation.  Notice   that~\eqref{eq:hyp-param}  and~\eqref{eq:variation}
imply   that  $\psi$   is   strictly  convex   and   a  bijection   from
$[\theta_0,     \infty    )$     to     $[0,     \infty    )$,     with
$\theta_0\in  [0,+\infty) $  the  largest root  of  $\psi=0$; and  that
$\psi$   is   (sub-)critical  if   and   only   if  $\theta_0=0$.   So,
under~\eqref{eq:variation},    we    denote     for    $\lambda>0$    by
$\psi^{-1}(\lambda)$ the only root in $(0, \infty )$ of $\psi=\lambda$.

\item The process $\bY$ and the branching mechanism are  conservative (that is, $\P^\psi_x(Y_t<+\infty
  )=1$ for all $t\geq 0$ and $x\in \R_+$):
   \begin{equation}
    \label{eq:conservative} \tag{H2}
\boxed{    \int_{0+} \frac{\rd \lambda} {|\psi(\lambda)|}= +\infty .}
\end{equation}
Under~\eqref{eq:hyp-param}, if  $\int_{[1, \infty )} r\, \pi(\rd r)$ is finite, then
the branching mechanism 
 $\psi$   is  conservative. 
Under~\eqref{eq:conservative}, the lifetime of $\bY$ is defined by   $\zeta=\inf\{t\,     \colon\,
Y_t=0\}$ with the convention that $\inf \emptyset=+\infty $.

\item The Grey condition:
  (which we shall consider with~\eqref{eq:hyp-param} and~\eqref{eq:variation}):
  \begin{equation}
    \label{eq:grey} \tag{H3}
 \boxed{   \int^{+\infty}\frac{\rd\lambda}{\psi(\lambda)}<+\infty.}
\end{equation}
Under~\eqref{eq:hyp-param}-\eqref{eq:conservative}, the Grey condition is equivalent to a.s.\
either the lifetime   $\zeta$   is     finite     or
$ \lim_{t\rightarrow+\infty } Y_t=+\infty$.

\item (Sub-)critical branching:
   \begin{equation}
    \label{eq:critical} \tag{H4}
\boxed{  \psi'(0)\geq  0.}
\end{equation}
  Under~\eqref{eq:hyp-param}, this condition is equivalent to 
$\lim_{t\rightarrow +\infty } Y_t=0$  a.s., and it
implies~\eqref{eq:conservative}. 
\end{itemize}

We        shall       most        of        the       time        assume
that~\eqref{eq:hyp-param}-\eqref{eq:grey}  hold, and  add some  comments
when~\eqref{eq:critical} does not hold, that is, for the super-critical case.

\medskip

Under~\eqref{eq:hyp-param}-\eqref{eq:grey}, we  define the  function $c$  on $(0,+\infty)$  as the  unique
nonnegative solution of the equation:
\begin{equation}
   \label{eq:def-v}
   \int_{c(h)}^{+\infty}\frac{\rd \lambda}{\psi(\lambda)}=h
   \quad\text{for $h>0$.}
 \end{equation}
We have $c=\lim_{\lambda \rightarrow +\infty } v(\lambda, \cdot)$ and the distribution of the lifetime $\zeta$ of  $\bY$ is
given by, for all $x\geq 0$:
\[
  \P^\psi_x(\zeta\leq h)=\expp{-x c(h)}
  \quad\text{for all $h\geq 0$.}
\]

\subsection{The Lévy tree}
\label{sec:LT}
The genealogy of the CB process $\bY$ (under the canonical measure) can
be described by a  random tree, the so called Lévy tree, see~\cite{Duquesne2002}. The next
description is from  \cite{Duquesne2005a}  using  the
coding of compact real trees by  height function (the considered
(sub-)critical  case  can be extended to the
super-critical case using a Girsanov transformation,  see
\cite{Abraham2011a}).
Under~\eqref{eq:hyp-param}-\eqref{eq:grey}, there  exists a $\sigma$-finite measure $\N^\psi[\rd\ct]$
on $\T$,  or excursion  measure of a  L\'{e}vy tree (with distance $d$
and root $\root$),
satisfying   the following
properties.
\begin{enumerate}[(i)]
\item\label{item:height}  \textbf{Height $\bH(\ct)$.} The distribution of the
  height  $\bH(\ct)$ is given by:
  \begin{equation}
   \label{eq:hauteur}
   \N^\psi[\bH(\ct)=0]=0
\quad\text{and}\quad
   \N^\psi[\bH(\ct)>h]=c(h)
   \quad\text{for all $h>0$.}
  \end{equation}

   \item \textbf{Local time.}
There exists a $\ct$-measure valued process $(\ell^a)_{a\geq 0}$
c\`adl\`ag for the weak topology on finite measures on $\ct$ such that:

\begin{enumerate}[-]
\item        $\ell^0=0$        and,       for        every        $a>0$,
  $\{\ell^a \neq 0\}=\{\bH(\ct)>a\} $, $\N^\psi[d\ct]$-a.e..
\item For every $a>0$, $\ell^a$ is supported on $\ct(a)=\{x\in \ct\,
  \colon\, d(\rho,x)=a\}$ (see~\eqref{eq:t(h)})
  and $\ell^a (\ct\setminus\Lf(\ct))=0$,
  $\N^\psi[d\ct]$-a.e.. 

\item For every $a>0$, we have $\N^\psi[d\ct]$-a.e.\ for
 every bounded
continuous function $f$ on $\ct$ and with the measure $\cn_a^{\ct}$
defined in~\eqref{eq:def-cna}:
\begin{align*}
\langle\ell^a, f \rangle
& = \lim_{\varepsilon \downarrow 0}
\frac{1}{c(\varepsilon)} \int f(x) \ind_{\bH(\ct')\ge
\varepsilon\}} \cn_a^{\ct}(\rd x, \rd\ct') \\
 & = \lim_{\varepsilon \downarrow 0} \frac{1}{c(\varepsilon)} \int f(x)
\ind_{\{\bH(\ct')\ge \varepsilon\}}
\cn_{a-\varepsilon}^{\ct}(\rd x, \rd\ct'),\ \text{if}\ a>0.
\end{align*}
\item The process $(\langle \ell^a, 1 \rangle)_{a\geq 0}$ is distributed
  as the CB process $\bY$ under its canonical measure with branching
  mechanism $\psi$. For simplicity, we shall identify the two processes:
  $Y_a=\langle \ell^a, 1 \rangle$. 
\end{enumerate}

In particular, we have for $\lambda>0$ (or $\lambda\geq 0$ if $\psi'(0)>-\infty $), $a\geq 0$ and 
$v$ defined by~\eqref{eq:Lap-Y:v} that:
\begin{equation}
  \label{eq:N(1-eY)}
  \N^\psi \left[1- \expp{-\lambda Y_a}\right]=v(\lambda, a),
\end{equation}
and with $\theta_0$ the largest nonnegative root of $\psi=0$, and thus
$v(\theta_0, a)=\theta_0$, that:
\begin{equation}
  \label{eq:N(1-eY)0}
  \N^\psi \left[1- \expp{-\theta_0 Y_a}\right]=\theta_0. 
\end{equation}

\item  \textbf{Branching property.} For every $a>0$, the conditional distribution of $\cn_a(\rd x, \rd
  \ct')$ under
  $\N^\psi[\cdot\, |\, \bH(\ct)>a]$ given $R_a(\ct)$ is that of a
  Poisson point measure on $\ct(a)\times \T$ with intensity $\ell^a(\rd
  x)\N^\psi[\rd \ct']$. 
    
\item \textbf{Mass measure and total mass $\sigma$.} The mass measure $\bm$ defined by:
\begin{equation}
\label{eq:int-la}
\bm(\rd x) = \int_0^\infty \ell^a(\rd x) \, \rd a,
\end{equation}
is supported by the leaves:  $\bm(\ct \setminus \Lf(\ct))=0$. 
The distribution of the  total mass $\sigma=\bm(\ct)$ 
is given by:
\begin{equation}
   \label{eq:N[1-S]}
  \N^\psi[\sigma=0]=0
  \quad\text{and}\quad
  \N^\psi[1- \expp{-\lambda \sigma}]=\psi^{-1}(\lambda)
  \quad\text{for $\lambda> 0$.}  
\end{equation}
In particular, we have  (with $\theta_0$ the largest root of $\psi=0$) that:
\[
  \N^\psi[\sigma=+\infty
  ]=\psi^{-1}(0+)=\theta_0.
\]
If $\psi$ is (sub-)critical, then
the Lévy tree $\ct$ is compact and $\sigma$ is finite $\N^\psi$-a.e.\ and:
\[
   \N^\psi[\sigma]=\inv{\psi'(0)}\in (0, +\infty ].
\]

\item \textbf{Branching points.}
\begin{itemize}
\item $\N^\psi[d\ct]$-a.e., for all $x\in \ct$, we have $\kappa_x\in
  \{1, 2, 3, \infty \}$ and $\kappa_\root=1$;  in particular the branching points of $\ct$
  have degree 3 (corresponding to  2
  children)  or   infinity.
\item  $\N^\psi$-a.e.\ the set $\{x\in \ct\,\colon\, \kappa_x=3\}$ is empty
   if $\beta=0$, or a countable dense subset of $\ct$ if $\beta>0$. 
\item The set $\Br_\infty(\ct)$  of infinite branching points is
  nonempty with $\N^\psi$-positive measure if and only if $\pi\ne 0$. If
  $\langle  \pi,1\rangle=+\infty$, the set  $\Br_\infty(\ct)$ is
  $\N^\psi$-a.e.\ a countable dense subset of $\ct$.
   If
  $\langle  \pi,1\rangle<+\infty$, the set  $\Br_\infty(\ct)$ is
  $\N^\psi$-a.e.\ a finite (possibly empty) subset of $\ct$.
\end{itemize}
\item \textbf{``Size'' of the vertices.}
The set $\{ d(\root,x),\ x\in \Br_\infty(\ct) \}$ coincides
$\N^\psi$-a.e.\ with the set of discontinuity times of the mapping $a\mapsto
\ell^a$. Moreover, $\N^\psi$-a.e., for every such discontinuity time $b$
of the map $a \mapsto \ell^a$, there
is a unique vertex $x_b\in \Br_\infty(\ct)\cap\ct(b)$ and
$\Delta_b>0$, such that:
\[
\ell^b = \ell^{b-} + \Delta_b \delta_{x_b},
\]
where the ``size'' $\Delta_b$ of the vertex $x_b$ is positive and  can  be obtained
by the approximation:
\begin{equation}
\label{DefMas}
\Delta_b = \lim_{\varepsilon \rightarrow 0}
\frac{1}{c(\varepsilon)}
n(x_b,\varepsilon),
\end{equation}
with $n(x_b,\varepsilon)=\int \ind_{\{x_b\}}(x)\ind_{\{\bH(\ct') >
\varepsilon\}} \cn_b^\ct(\rd x,\rd\ct')$ is the number of sub-trees
above level $b$ with  root
$x_b$ and height larger than $\varepsilon$.
\end{enumerate}

In order to stress the dependence in $\ct$, we may write $\ell^{a, \ct}$
for the local time $\ell^a$, $\bm^\ct$ for the mass measure  and  $\sigma^\ct$  for the total mass.

\begin{rem}[On Assumptions~\eqref{eq:hyp-param}-\eqref{eq:critical}
  for the Lévy tree]
  \label{rem:onH1-4}
  The Lévy trees are defined under~\eqref{eq:hyp-param}-\eqref{eq:grey},
  but they  could be  introduced without some  of those  assumptions. If
  Assumption~\eqref{eq:variation} does  not hold, then the  Lévy tree is
  discrete,   and  the   machinery  developed   here  is   not  adapted.
  Assuming~\eqref{eq:hyp-param}-\eqref{eq:conservative},      the      Grey
  condition~\eqref{eq:grey} implies  that the  Lévy tree belongs  to the
  Polish space $\T$ of (GH-isometric  classes) of locally compact rooted
  tree;  using the  mass erasure  procedure approach  from Duquesne  and
  Winkel~\cite{duquesne2022mass},  it  seems  possible  to  get  rid  of
  Assumption~\eqref{eq:grey} by  considering a larger space  of trees
  with    nice    topological    properties.     Under~\eqref{eq:hyp-param}-\eqref{eq:grey},    the
  condition~\eqref{eq:critical}  is then  equivalent  to  the Lévy  tree
  being $\N^\psi$-a.e.\ compact, that is, having a finite height.
\end{rem}

\subsection{Girsanov transformation and related measure on L\'evy trees}
\label{sec:meas-LT}
Assume~\eqref{eq:hyp-param}-\eqref{eq:grey}. 
We  define a  probability  measure on  $\T$  as follow.   Let $r>0$  and
$\sum_{i\in\ci}\delta_{(h_i,\ct_i)}(\rd h, \rd \ct)$ be  a Poisson point measure  on
$\R_+\times \T$ with
intensity  $\rd h \N^\psi[\rd \ct]$.   Consider $\{\root\}$  as  the trivial  measured
rooted  real tree reduced  to the  root. We define for $r>0$:
\[
  \ct_{[r]}=\{\root\}   \circledast_{i\in  \ci_r}(  \ct_i,   \root)
  \quad\text{with}\quad
  \ci_r=\{i\in \ci\, \colon\, h_i\leq  r\}.
\]
Since $\N^\psi[H(\ct)=0]=0$ and $\N^\psi[H(\ct)>\varepsilon]$ is finite
for all $\varepsilon>0$,
we deduce from Section~\ref{sec:T-T1} that $\ct_{[r]}$ is a $\T$-valued
random variable.  We denote by $\P^\psi_r$ its probability
distribution.  Notice that $\P^\psi_r$-a.s., the root is an infinite
branching vertex and its size $\Delta_\root$ defined by~\eqref{DefMas}
(with $b=0$)
is exactly  $r$. 
The corresponding local 
time  is defined  by $\ell^a_{[r]}=\sum_{i\in
  \ci_r}  \ell^{a, \ct_i}$  for  $a> 0$,  and we set
$\ell^0_{[r]}=r \delta_\root$. Thanks to 
Property (ii) on the local time, the process $
(\langle \ell^a_{[r]} , 1\rangle_a\geq 0)$ is distributed as the CB process
$\bY$  under $\P^\psi_r$.

\medskip

We now recall the  Girsanov
transformation from~\cite{Abraham2011a,Abraham2011}; this transformation allows in particular to define the Lévy trees in the super-critical regime. 
Assume~\eqref{eq:hyp-param}.
We consider the interval $\Theta'=\{\theta\in \R\, \colon\, \int_{(1, +\infty )}
\expp{-\theta r} \, \pi(\rd r)<+\infty \}$, and notice that  $\R_+
\subset \Theta$. The
function $\psi$ given by~\eqref{eq:psi} is in fact well defined on
$\Theta'$. 
For $\theta\in \Theta'$, we define the function
$\psi_\theta$ by:
\begin{equation}
   \label{eq:def-psi-q}
  \psi_\theta(\lambda)= \psi(\lambda+\theta) - \psi(\theta)
  \quad\text{for $\lambda\geq 0$;}  
\end{equation}
it is a branching mechanism whose 
quadratic parameter  is the same $\beta_\theta=\beta$, and whose  Lévy measure
is given by $\pi_\theta(\rd r)= \expp{-r  \theta} \, \pi(\rd r)$.
For $\theta\in \Theta'$, 
the       branching        mechanism       $\psi_\theta$       satisfies
also~\eqref{eq:hyp-param}.
Eventually, we consider the interval $\Theta^\psi=\{\theta\in \Theta'\,
\colon\,  \psi_\theta 
\text{ is conservative}\}$, and it is elementary to check that
$\Theta^\psi\subset \Theta' \subset  \Theta ^\psi \cup \{\theta_\infty\}$, with
$\theta_\infty= \inf \Theta'\in [-\infty , 0]$.
If  furthermore   the  branching  mechanism
$\psi$ satisfies~\eqref{eq:variation}  (resp.~\eqref{eq:grey}), then the
branching  mechanism  $\psi_\theta$ satisfies  also~\eqref{eq:variation}
(resp.~\eqref{eq:grey}).

For $h\geq 0$, we set $\cf_h$  the  $\sigma$-field  on $\T$ generated  by the
truncation map $R_h$, and under $\N^\psi[\rd \ct]$, we denote 
$\sigma_h=\bm^\ct  (\ct_h)=\int_0^h Y_a\, \rd a$ the total mass of the
truncated Lévy tree  $\ct_h=R_h(\ct)$, with 
 $Y_a=\langle
\ell^a, 1 \rangle$ the total local
time at level $a$ of $\ct$. 
It is elementary to deduce from~\cite[Theorem~2.22]{Abraham2011} (see
also~\cite[Lemma~3.8 and Corollary~4.4]{Abraham2009b}) that for
$\theta\in \Theta^\psi$ and 
$h\geq 0 $:
\begin{equation}
   \label{eq:Girsanov}
 \frac{d\N^{\psi_\theta}}{d\N^\psi}_{|\cf_h}=  \expp{ -\psi(\theta) \sigma_h - \theta  Y_h},
\end{equation}
and furthermore, with $ \sigma= \bm^{\ct}(\ct)$ the total mass of $\ct$:
\begin{equation}
   \label{eq:Girsanov-sigma}
 \ind_{\{\sigma<+\infty \}}\, \N^{\psi_{\theta}}[\rd \ct]
 = \ind_{\{\sigma<+\infty \}}\,\expp{ - \psi(\theta) \sigma}\, 
 \N^\psi[\rd \ct]. 
\end{equation}
Recall that 
$\N^{\psi_\theta}$-a.e.\ $\sigma$ is finite if and only if 
$\psi_\theta$ is sub-critical or critical. 
We deduce   that the distribution of the total mass $\sigma$ under $\N^\psi$
on $(0, +\infty )$ has a density w.r.t.\ the Lebesgue measure, say
$f_\sigma$, if and only if  the distribution of  $\sigma$ under $\N^{\psi_\theta}$
on $(0, +\infty )$ has a density w.r.t.\ the Lebesgue measure, say
$f^\theta_\sigma$, and that in this case:
\begin{equation}
   \label{eq:density-q}
  f^\theta_\sigma(r)=\expp{- \psi(\theta)r }\, f_\sigma(r)
  \quad\text{for}\quad
  r\in (0, +\infty ).
\end{equation}

\medskip

Let $\theta\geq 0$. 
We will consider the following measure  $\cn_\theta^\psi
[d\ct]=2\beta\theta\N^{\psi}[\rd\ct]
 +\int_{(0,+\infty )}  {\pi}( \rd r)(1-\expp{-\theta r})\P_r^{\psi}(\rd\ct)
$ on $\T$ and its formal derivative $\partial_\theta \; \cn_\theta^\psi
[d\ct]$  at $\theta=0$, that is:
\begin{equation}
\label{eq:def-bN}
\bN^\psi[d\ct]=2\beta \N^\psi[\rd\ct]+\int_{(0,+\infty)}r\pi(\rd r)\,\P_r^\psi(\rd\ct).
\end{equation}
Elementary computations yield for $q>0$ such that $\psi(q)>0$:
\begin{equation}
   \label{eq:bN-moment}
\bN^\psi\left[1-\expp{-\psi(q) \sigma}\right]=2\beta q + \int_{(0,
  +\infty )} (1- \expp{-q r}) \, r \pi(\rd r) = \psi'(q)
- \psi'(0).
\end{equation}

\subsection{The Kesten tree $\ck$}
\label{sec:kesten}
Recall the infinite branch $\spine$ defined at the end of
Section~\ref{sec:T-T1}. Its length measure, denoted by  $\rd h$, is just the Lebesque measure
on $\R_+$.

\begin{defi}[Kesten tree]
  \label{defi:kesten}
  Let  $\psi$   be  a   branching  mechanism  satisfying~\eqref{eq:hyp-param}-\eqref{eq:grey}.   Let
  $(h_i, \ct_i)_{i\in  I}$ be the  atoms of  a Poisson point  measure on
  $\spine \times  \T$ with intensity  $\rd h \, \bN^\psi[\rd  \ct]$ (and
  $\bN^\psi$ given in~\eqref{eq:def-bN}).  The  Kesten tree $\ck$, whose
  distribution is  denoted $\P^\psi$,  is a $\T$-valued  random variable
  defined as the infinite branch (or  spine) $\spine$, on which the $\ct_i$ are
  grafted at height $h_i$, that is:
\[
  \ck= \spine  \circledast_{i\in
    I}( \ct_i,h_i).
\]
\end{defi}
The left part of Fig.~\ref{fig:height} give an illustration of the Kesten tree. 
Notice  the  Kesten  tree  is   indeed  a  $\T$-valued  random  variable
by~\cite[Lemma~5.31]{adhe-2022}.
We see  $h\geq 0$ as an element of the spine $\spine$ at distance $h$
from the root $\root=0$ and as a distinguished vertex of $\ck$. In particular the
   Kesten tree whose spine is
truncated at level $h\geq 0$, that is, $\wck_h=\wbr(\ck, h)$, can be described as: 
\begin{equation}
   \label{eq:Kh}
  \wck_h=\II{0,h}  \circledast_{i\in
    I_h}( \ct_i,h_i)
  \quad\text{with}\quad
  I_h=\{i\in I\, \colon \, h_i\leq h\},
\end{equation}
and
$(\wck_h, h)$ is a $\T_1$-valued random variable.
The  local  time of the  Kesten tree
 $\ck$      at     level      $h\geq      0$      is     defined      by
 $\ell^{h,   \ck}=\sum_{i\in   I_h}  \ell^{h-h_i,   \ct_i}$.

We can recast Theorem~4.5 from~\cite{Duquesne2005a} in our setting for
the Lévy tree with a marked leaf (see~\cite[Corollary~5.9]{adhe-2022}
for the supercritical quadratic, that is, when $\pi=0$). 
\begin{prop}[Lévy tree with a marked leaf at a given level]
   \label{prop:TX=Kh}
   Assume~\eqref{eq:hyp-param}-\eqref{eq:grey}          hold         and
   $\psi'(0)> -\infty $.   Let $G$ be a  nonnegative measurable function
   defined on $\T_1$. We have for $h>0$:
\begin{equation}
   \label{eq:Levy-Kh}
     \N^\psi\left[\int_\ct \ell^h(\rd x)\, G(\ct, x)\right]
     = \expp{- h \psi'(0)} \E^\psi[G(\wck_h, h)].
   \end{equation}   
 \end{prop}
 \begin{proof}        For         the        (sub-)critical        case,
   see~\cite[Theorem~4.5]{Duquesne2005a}.       We      now       assume
   $\psi'(0)\in (-\infty ,  0)$ and we shall use a  Girsanov argument to
   deduce the result from the sub-critical case. Let $h>0$. 
  
   For a pointed tree $(T,d, \root,  x)\in \T_1$, where $\root, x\in T$,
   we  uniquely  decompose   the  tree  $T$  according   to  the  branch
   $\II{\root, x}$ and the sub-trees grafted on this branch:
  \[
    T= \II{\root, x}\circledast_{i\in I}(T_i,x_i),
  \]
  where $x_i$  belongs to the  branch $\II{\root,  x}$ and $T_i$  is the
  union  of all  the connected  components of  $T\setminus \{x_i\}$  not
  containing the root  nor $x$, with $x_i$ added as  a root. Notice this
  decomposition   is  unique   (and  measurable   using  an   adaptation
  of~\cite[Proposition~5.32]{adhe-2022}).     We   shall    consider   a
  particular     choice      of     function     $G$      defined     by
  $ G(T,x)=\exp{\left( - \sum_{i\in I} G_0(d(\root, x_i), T_i)\right)}$,
  with   $G_0$   a   nonnegative    measurable   function   defined   on
  $\R_+\times\T $ such that $G_0(T,x)=0$ if $\bH(T)>a-h$ for some $a$ large.

  Let   $\theta_0>0$  be   the  only   root  of   $\psi=0$.  We   deduce
  from the Girsanov transformation~\eqref{eq:Girsanov} that:
 \begin{align*}
   \N^\psi\left[\int_\ct \ell^h(\rd x)\, G(\ct, x)\right]
  & =  \N^{\psi_{\theta_0}}\left[\int_\ct \ell^h(\rd x)\, G(\ct,
    x)\expp{\theta_0 Y_a}\right]\\
   &= \expp{-h \psi'(\theta)} \E^{\psi_\theta} \left[G(\wck
     _h,h)\expp{\theta_0 \ell^{h, \ck}}\right]\\
  &=   \expp{-h \psi'(\theta)}
    \exp{\left(- \int_0^h \rd s \, \bN^{\psi_{\theta_0}} [1- \expp{-G_0(s,
    \ct)+ \theta_0 Y_{a-s}}]\right)}.
 \end{align*}
Using~\eqref{eq:Girsanov} and then~\eqref{eq:N(1-eY)0}, we get that:
 \[
   \N^{\psi_{\theta_0}} [1- \expp{-G_0(s,     \ct)+ \theta_0 Y_{a-s}}]
   =  \N^{\psi} [1- \expp{-G_0(s,\ct)}] -   \N^{\psi} [1- \expp{
     -\theta_0 Y_{a-s}}]
      =  \N^{\psi} [1- \expp{-G_0(s,\ct)}] -\theta_0. 
\]
Using that the Lévy measure associated to $\psi_{\theta_0}$ is
$\expp{-\theta_0 r }\pi(\rd r)$, we then deduce that:
\begin{multline*}
  \int_{(0, +\infty )} r\expp{-\theta_0 r}\pi(\rd r)
  \P_r^{\psi_{\theta_0}} \left(1- \expp{-G_0(s,     \ct)+ \theta_0 Y_{a-s}}
  \right)\\
  =  \int_{(0, +\infty )} r\, \pi(\rd r)
  \P_r^{\psi} \left(1- \expp{-G_0(s,     \ct)}
  \right)
- \int_{(0, +\infty )} (1-\expp{-\theta_0 r})\, r\pi(\rd r).
\end{multline*}
Thus, we obtain that:
\[
  \bN^{\psi_{\theta_0}} [1- \expp{-G_0(s,
    \ct)+ \theta_0 Y_{a-s}}]
  =\bN^{\psi} [1- \expp{-G_0(s,
    \ct)}] - \psi'(\theta_0) + \psi'(0),
\]
and then that:
\[
     \N^\psi\left[\int_\ct \ell^h(\rd x)\, G(\ct, x)\right]
=  \expp{-h \psi'(0)} \E^{\psi} \left[G(\wck
  _h,h)\right].
\]
By dominated  convergence, we can remove the condition that $G_0(T,x)=0$
if $\bH(T)>a-h$ for some $a$ large, and then by the monotone class
theorem, we deduce the previous equality holds for any 
       nonnegative measurable function $G$
   defined on $\T_1$. 
\end{proof}

 Taking $G\equiv 1$ in~\eqref{eq:Levy-Kh}, if $\psi'(0)>-\infty $, we
 get that for $h>0$:
 \begin{equation}
   \label{eq:NYh}
     \N^\psi\left[\langle  \ell^h, 1 \rangle\right]
     =\N^\psi\left[Y_h\right]=\expp{- h\psi'(0)}.
 \end{equation}
 For the  sub-critical case, inegrating~\eqref{eq:Levy-Kh} over  $h$ and
 using~\eqref{eq:int-la}, we also obtain the following corollary.

\begin{cor}[Sub-critical Lévy tree with a marked leaf]
   \label{cor:TX=Kh}
   Assume~\eqref{eq:hyp-param}-\eqref{eq:grey}  hold and  $\psi'(0)>0$.  Let $G$ be
   a nonnegative measurable function defined on $\T_1$. 
   We have:
   \[
 \inv{  \N^\psi[\sigma]}\,  \N^\psi\left[\int_\ct \bm(\rd x)\, G(\ct, x)\right]
     =  \E^\psi[G(\wck_{E}, E)],
   \]
   where $E$ is under $\P^\psi$  an exponential random variable with mean
   $\N^\psi[\sigma]=1/\psi'(0)$ independent of  $\ck$.
 \end{cor}

   Provided
 $\psi'(0)>-\infty       $,        the       total local time
 process of the Kesten tree, 
 $(\langle  \ell^{a,  \ck},  1  \rangle)_{a\geq  0}$,  is  a  CB  process
 $\bZ=(Z_a)_{a\geq 0}$ with immigration (CBI  process) defined as the CB
 process    with   branching    mechanism    $\psi$   and    immigration
 $\psi'   -   \psi'(0)$.    As   a   consequence   of~\eqref{eq:Levy-Kh}
 and~\eqref{eq:NYh}, if $\psi'(0)>  -\infty $ the one-dimensional
 marginal  $Z_a$  is distributed  as  the  size biased  distribution  of
 $Y_a=\langle \ell^a, 1  \rangle$ under the excursion  measure, that is,
 for $a>0$ and $g$ a measurable nonnegative function defined on $\R_+$:
\begin{equation}
  \label{eq:EZK}
  \E^\psi[g(Z_a)]=\expp{ a\psi'(0)}\N^\psi[Y_a
  \,   g(Y_a)] =\frac{\N^\psi[Y_a\, 
    g(Y_a)]}{\N^\psi[Y_a]} \cdot
\end{equation}

\section{L\'evy tree with a given height}

\subsection{Coupling and convergence for the conditioning by the height}
Assume~\eqref{eq:hyp-param}-\eqref{eq:critical} so  that  $\psi'(0)=\alpha\geq 0$  and  the branching  is
(sub-)critical.      We refer to Remark~\ref{rem:super-crit-H} for the
super-critical case where~\eqref{eq:critical} is not satisfied. 
According         to~\eqref{eq:def-v}        and
Item~\ref{item:height}  in  Section~\ref{sec:LT},   we  get  that  under
$\N^\psi[\rd \ct]$ the distribution of the height $\bH(\ct)$ of the Lévy tree
$\ct$   has   the  density   $|c'|$   w.r.t.\   the  Lebesque   measure.
By~\cite[Corollary~1.29]{k:rm}, there  exists a  regular version  of the
conditional     distribution    of     $\ct$    given     its    height,
$\N^\psi[\rd     \ct     \,      |\,     \bH(\ct)=h]$     for     $h>0$.
See~\cite[Theorem~3.3]{Abraham2009a}  for a  nice representation  of the
compact rooted random  tree $\cth$ distributed as  $\ct$ under $\N^\psi$
conditionally  on $\{  \bH(\ct)=h\}$ for  any given  $h>0$:
\[
  \cl(\cth)= \N^\psi[\rd \ct \, |\,  \bH(\ct)=h].
  \]
From this representation, 
there exists a unique vertex $X_h$ of  $\cth$ which is at distance $h$ of
the root.


Recall  the  Kesten  tree $\ck$  from  Definition~\ref{defi:kesten}.  We
consider the tree  $\ckh$ defined as the Kesten tree  whose spine is cut
at level $h$ and where the  trees $\ct_i$, grafted at height $h_i$ (less
that $h$) and with height larger than $h-h_i$ are removed, that is:
\[
  \ckh= \II{0, h}  \circledast_{i\in
    I^{\height}_h}( \ct_i,h_i)
  \quad\text{with}\quad
  I^{\height}_h=\{i\in I\, \colon\, H(\ct_i)+ h_i\leq  h\},
\]
and  the  branch $\II{0,h}$  can  be  seen  as  the restriction  of  the
semi-infinite spine  $\spine$ of the  Kesten tree  up to level  $h$.  We
shall also see $h$ as a distinguished vertex of $\ckh$;  it is the only one at distance $h$
from the root. (Let  us stress that $\ct_i$ is the  only tree grafted at
height  $h_i$ and  that  its  root is  an  infinite  branching point  if
$\kappa_{h_i}(\ck)=+\infty $.)   Let us  mention the  following relation
between    the    different    truncations     of    the    Kesten    tree:
$\ckh \subset R_h(\ck)\subset \wck=\wbr(\ck, h)$.

We have the following coupling of the Lévy tree conditioned on its
height and the Kesten tree whose proof is given in Section~\ref{sec:proof-H}.

\begin{theo}[Coupling for the height conditioning]
  \label{theo:height}
  Assume~\eqref{eq:hyp-param}-\eqref{eq:critical}  hold. Let $h\in (0, +\infty )$. Then, the Lévy tree
 under $\N^\psi$  conditioned to have height $h$, denoted
 $\cth$,  is distributed as the
  truncated Kesten tree $\ckh$, more precisely (in $\T_1$):
 \begin{equation}
   \label{eq:KH=Th}
(   \cth, X_h)
   \, \overset{\text{(d)}}{=} \,
   (\ckh,h).
  \end{equation} 
\end{theo}

\begin{rem}[Reconstruction of the L\'evy tree]
  \label{rem:recons-height}
If  $\psi$ is (sub-)critical  and $H$ is distributed as $\bH(\ct)$
under   $\N^\psi[\rd    \,   \ct]$    (that   is,    with   distribution
$|c'(h)|\, \rd h$ on $(0, +\infty )$) and independent of the Kesten tree
$\ck$ under $\P^\psi$, we  get that  $\ckH$ is  distributed as  the unconditioned
Lévy tree $\ct$ under $\N^\psi[\rd \ct]$.
\end{rem}

By construction, we have the following result on the coupled Kesten
sub-trees. Recall the  convention~\eqref{eq:conv-inclu}
and~\eqref{eq:conv-lim} for the inclusion and limit of trees.

\begin{prop}[Monotony and local convergence for the height coupling]
  \label{prop:height}
  Assume~\eqref{eq:hyp-param}-\eqref{eq:critical}  hold. We have  that for $h'\geq h>0$:
  \[
    \ckh \subset
    \ckhp \quad\text{a.s.,}
  \]
   and the following 
   convergence holds:
 \[
   \lim_{h\rightarrow +\infty} \ckh \stackrel{\text{a.s.}}{=}  \ck
     \quad\text{in}\quad
  (\T, \dlgh).
\]
\end{prop}

\begin{proof}
  The inclusion is clear.  Notice that $R_a(\ckh)=R_a(\ck) $ for all $a$
  strictly                           less                           than
  $A_h=\inf\{h_i\,   \colon\,   i\in   I  \quad\text{and}\quad   h_i   +
  \bH(\ct_i)>  h\}$.  Since  $\psi$ is  (sub-)critical by~\eqref{eq:critical},  we have
  $\N^\psi[\bH(\ct)\geq r]$ finite for all $r>0$. We deduce that the map
  $h\mapsto    A_h$   which    is    non-decreasing   satisfies    a.s.\
  $\lim_{h\rightarrow \infty } A_h=+\infty  $.  This 
  and~\eqref{eq:rh(t)} readily imply the
  a.s.\ convergence in $(\T, \dlgh)$.
\end{proof}

We immediately deduce from Theorem~\ref{theo:height}
and Proposition~\ref{prop:height} the following  well known result.

\begin{theo}[Local convergence for the height conditioning]
\label{theo:cv-loc-height}
  Assume~\eqref{eq:hyp-param}-\eqref{eq:critical}  hold. The following 
local   convergence holds in distribution:
 \begin{equation}
   \label{eq:limTh}
      \cth  \xrightarrow[h\rightarrow
   +\infty]{\text{(d)}}
      \ck
     \quad\text{in}\quad
  (\T, \dlgh).
 \end{equation}
\end{theo}

\begin{rem}[The super-critical case]
  \label{rem:super-crit-H}
  Assume~\eqref{eq:hyp-param}-\eqref{eq:grey}  hold and  that $\psi$  is
  super-critical.   Then,  according  to~\eqref{eq:Girsanov-sigma},  the
  Lévy tree $\ct$ conditioned to have  a finite mass under $\N^\psi$, or
  equivalently  a  finite  height  thanks  to  the  Grey  condition,  is
  distributed  as  the  Lévy tree  under  $\N^{\psi_{\theta_0}}$,  where
  $\theta_0   $   is   the   only  positive   root   of   $\psi=0$   and
  $\psi_{\theta_0}$ is defined by~\eqref{eq:def-psi-q}.  Furthermore, it
  is elementary to check  that the branching mechanism $\psi_{\theta_0}$
  defined   by~\eqref{eq:def-psi-q}   is  sub-critical.    Thus,   under
  $\N^\psi$,  the conditioned  Lévy  tree $\cth$,  which  is still  well
  defined,  is distributed  as the  conditioned Lévy  tree $\cth$  under
  $\N^{\psi_{\theta_0}}[\rd  \ct]$.   This  gives  that~\eqref{eq:KH=Th}
  and~\eqref{eq:limTh}    hold     if    $\psi$     is    super-critical
  under~\eqref{eq:hyp-param}-\eqref{eq:grey}, but with  $\ck$ the Kesten
  tree    associated   to    the   sub-critical    branching   mechanism
  $\psi_{\theta_0}$.
\end{rem}

\subsection{Proof of Theorem~\ref{theo:height}}
\label{sec:proof-H}
Let $h>0$ be fixed. Recall the function $c$ defined in~\eqref{eq:def-v} so that $c(h)=
\N^\psi[ \bH(\ct)>h]$.  
   According to~\cite[Theorem~3.3]{Abraham2009a}, the distribution of
   $\cth$ and $X_h$ is given by:
   \[
 \II{0, h}  \circledast_{i\in
   I'_h}( \ct'_i,h'_i)
 \quad\text{and}\quad
 h,
\]
where   $(h'_i, \ct'_i)_{i\in  I}$ are  the atoms of a  Poisson point measure
  on $[0,h] \times \T$ with intensity  $\rd h' \, \bN_{h'}^\psi[\rd \ct]$
  with:
  \[
\bN_{h'}^\psi[\rd \ct]= 2 \beta \N^\psi[\rd \ct, \, \bH(\ct)\leq
h-h']
+ \int_{(0, +\infty )} r \expp{- r c(h-h')} \pi(\rd r) \P_{r, h'}^\psi
(\rd \ct),
\]
and $\P_{r,h'}^\psi (\rd \ct)$ is the distribution of 
$\{\root\}   \circledast_{j\in  \cj}(  \ct''_j,   \root)$  where
$(\ct''_j)_{j\in  \cj}$ are  the atoms of a  Poisson point measure
on $ \T$ with intensity  $r \N^\psi[\rd \ct, \bH(\ct) \leq  h-h']$.
Notice that $X_h$ in $\cth$ and $h \in \ckh$ are
the only element at distance $h$ from the root of their tree.

So, the proof is complete once we get that 
$( h_i, \ct_i)_{i \in I^{\height}_h}$ and $(h'_i, \ct'_i)_{i\in  I'}$ have the
same distribution. Let $G$ be a nonnegative measurable function defined
on $\R_+\times \T$. We have:
\[
  \E\left[\expp{- \sum_{i\in I'} G(h'_i, \ct'_i)}\right]
= \exp\left\{ - \int_0^h \rd h'\bN_{h'}^\psi\left[1- \expp{-G(h', \ct)}\right]
  \right\}
\]
and:
\begin{multline*}
 \bN_{h'}^\psi\left[1- \expp{-G(h', \ct)}\right]
  =2 \beta \N^\psi \left[(1- \expp{-G(h', \ct)})\ind_{\{\bH(\ct) +h' \leq
    h\}}\right]\\
    + \int_{(0, +\infty )} r \expp{- c(h-h')} \pi(\rd r) \E_{r, h'}^\psi
\left[1- \expp{-G(h', \ct)}\right].
\end{multline*}
Now considering first a function $G$ such that $G\left(h', \{\root\}
\circledast_{j\in  \cj'}(  \ct^j,   \root)\right)= \sum_{j\in \cj}
G_0(h',\ct^j)$ with $G_0$
a nonnegative measurable function defined
on $\R_+\times \T$, we get:
\begin{align}
  \label{eq:compute!}
  \E_{r, h'}^\psi \left[ \expp{-G(h', \ct)}\right]
&  = \exp\left\{ - r \N^\psi\left[ \left(1-
        \expp{-G_0(h',\ct)}\right)\ind_{\{\bH(\ct) + h'\leq  h\}}
  \right]\right\}\\
  \nonumber
&  = \exp\left\{ - r \N^\psi\left[ 1-
        \expp{-G_0(h',\ct)}\ind_{\{\bH(\ct) + h'\leq  h\}}
  \right] + r \N^\psi[ \bH(\ct) >h-h']\right\}\\
  \nonumber
&  = \E_{r}^\psi\left[ \expp{-G(h',\ct)}\ind_{\{\bH(\ct) + h'\leq  h\}}
  \right] \expp{ r c(h-h')}.
\end{align}
By the monotone class theorem, we deduce that for all nonnegative
measurable function $G$:
\begin{equation}
   \label{eq:Ph'=ecPh}
 \expp{ -r c(h-h')} \, \E_{r, h'}^\psi \left[ \expp{-G(h', \ct)}\right]
=  \E_{r}^\psi\left[ \expp{-G(h', \ct)}\ind_{\{\bH(\ct) + h'\leq  h\}}
\right]
\end{equation}
and thus:
\begin{align*}
  \bN_{h'}^\psi[\rd \ct]
   &= 2 \beta \N^\psi[\rd \ct, \, \bH(\ct)+h'\leq
h]
+ \int_{(0, +\infty )} r \pi(\rd r) \P_{r}^\psi
  (\rd \ct, \, \bH(\ct)+h'\leq h)\\
&=   \bN^\psi[\rd \ct, \, \bH(\ct)+h'\leq h] .
\end{align*}
By construction the left hand-side of the above equality times $\rd h'$
is the intensity of the Poisson point measure $\sum_{i\in I^{\height}_h}
\delta_ {(h_i , \ct_i)}$. We have thus obtained that 
$( h_i, \ct_i)_{i \in I^{\height}_h}$ and $(h'_i, \ct'_i)_{i\in  I'}$ have the
same distribution, which completes the proof.

\section{Lévy tree with  a given maximal vertex ``size''}

\subsection{Coupling and convergence for the conditioning by the maximal
  vertex size}

We assume~\eqref{eq:hyp-param}-\eqref{eq:critical}  and that the Lévy measure $\pi$ is non trivial and denote by 
$\supp(\pi)$ its closed support (in $(0, +\infty )$). In particular
$\psi'(0)\geq 0$. 
For $x\in \Br_\infty(\ct)$, we define the ``size'' of the vertex at $x$ by
$\Delta_x$ given as the right hand-side of~\eqref{DefMas}. Recall that
$\Br_\infty(\ct)$ is $\N^\psi$-a.e.\ a countable dense (resp.\ finite) subset of $\ct$
if $\langle \pi, 1 \rangle=\infty $ (resp.\  $\langle \pi, 1 \rangle<\infty $),
and that the set of vertices with size larger than $\varepsilon>0$ is finite
for any $\varepsilon>0$. 
We consider the maximal ``size'' of the vertices in the Lévy tree $\ct$ under $\N^\psi$:
\[
  \Delta= \max \{ \Delta_x\, \colon\, x\in \Br_\infty(\ct)\},
\]
with the  convention $\max  \emptyset=0$. If  necessary, we  shall write
$\Delta(\ct)$ to  stress the  dependence in  the tree  $\ct$.  According
to~\cite[Proposition~3.1]{nassif2022} (where the size of a vertex is
called mass therein),  the   distribution  of  $\Delta$
under $\N^\psi$ is given by, for $\delta>0$:
\begin{equation}
   \label{eq:ND>d}
  \N^\psi[\Delta>\delta]=\psi_{\{\delta\}}^{-1}(\bar \pi(\delta)),
\end{equation}
where $\bar \pi(\delta)$ is defined in~\eqref{eq:def-pid} and  for $\lambda\geq 0$:
\begin{equation}
   \label{eq:def-psi-d}
  \psi_{\{\delta\}}(\lambda)= \psi(\lambda)+ \int_{(\delta, \infty )} \left(
    1- \expp{-\lambda r}\right)\, \pi(\rd r).
\end{equation}
Intuitively,    the     function    $\psi_\delta$,     with
characteristic 
$(\alpha_{\{\delta\}},  \beta_{\{\delta\}},  \pi_{\{\delta\}})$, is  the
branching mechanism  of the CB  process with branching  mechanism $\psi$
after removing all the jumps of  size strictly larger than $\delta$ (and
killed when  it reaches 0); notice  the quadratic parameter is  the same
$\beta_{\{\delta\}}=\beta$,    the   Lévy    measure    is   given    by
$\pi_{\{\delta\}}(\rd r)=  \ind_{\{r \leq \delta\}} \pi(\rd  r)$ and the
drift                   is                    given                   by
$\psi'_{\{\delta\}}(0)=\alpha_{\{\delta\}}=\psi_\theta'(0)+         \int
_{(\delta,     \infty    )}     r     \,     \pi(\rd    r)$.      Thanks
to~\cite[Corollary~3.2]{nassif2022},         we         get         that
$\N^\psi[\Delta\in  \cdot]$   and  $\pi$   have  the  same   support  in
$(0, \infty  )$, and  the same  atoms in  $(0, \infty  )$, and  thus for
$\delta>0$:
\begin{equation}
   \label{eq:pi(d)}
  \pi(\{\delta\})=0 \quad \Longleftrightarrow\quad
  \N^\psi[\Delta\geq \delta; \,  \cdot\, ]=
  \N^\psi[\Delta> \delta; \, \cdot\,].
\end{equation}

We follow~\cite{nassif2022} (notice   \eqref{eq:grey} is  not assumed
therein, but  one need them here  for $\ct$ to be  locally compact). Let
$\delta>0$     and      consider     a    random      marked     tree
$(\ctnr,  X_\delta)\in   \T_1$  whose   distribution is characterized by,  for  any
nonnegative measurable function $G$ defined on $\T_1$:
\begin{equation}
   \label{eq:EGTd}
  \E\left[G\left(\ctnr, X_\delta
  \right)\right]
=
\inv{\N^\psi[\sigma \ind_{\{\Delta\leq  \delta\}}]}
\N^\psi\left[ \ind_{\{
    \Delta\leq  \delta\}}\int_{\ct} \bm(\rd x) \, G(\ct, x)\right].
\end{equation}
Then define the random marked 
tree $(\ctn, X_\delta)\in \T_1$ by:
\begin{equation}
   \label{eq:Tnode*}
  \ctn=\ctnr \circledast (\ct'_{\{\delta\}},X_\delta),
\end{equation}
where   $\ct'_{\{\delta\}}$  is   independent  of   $(\ctnr,  X_\delta)$   and
distributed according to $\P^\psi_\delta( \cdot |\, \Delta\leq \delta)$.
By  construction, if  $\delta$ is  not  an atom  of $\pi$,  we get  that
$X_\delta$ is  the only vertex  of $ \ctn$  with size $\delta$  and that
$\Delta=\delta$. We thus get   $\wbr(\ctn,  X_\delta)=\ctnr$,   with  the
restriction  map $\wbr$  defined in~\eqref{eq:def-wbr}.   We denote  the
probability          distribution          of         $\ctn$          by
$\N^\psi[\rd \ct \, |\, \Delta=\delta]$; it is defined even for $\delta$
not  in   the  support  of   the  distribution  of  $\Delta$.    By  the
measurability of the grafting procedure, see~\cite{adhe-2022}, we deduce
that  the map  $\delta \mapsto  \N^\psi[\rd \ct  \, |\,  \Delta=\delta]$
(from $\R_+^*$  to the set of  probability measures on the  Polish space
$\T$)  is measurable.   According to~\cite[Theorem~5.7]{nassif2022},  if
$\pi\neq    0$     and $\pi$     \textbf{has    no     atom},    we
have:
\[
\int    _{\delta>0}    \N^\psi[\rd    \ct   \,    |\,    \Delta=\delta]
\,\N^\psi[\Delta=\rd \delta]=\N^\psi[\rd \ct].
\]
For this
reason, if $\pi \neq 0$ has no atom, 
we shall  say that $ \ctn$ \textbf{is distributed  as the L\'evy
  tree  $\ct$ conditioned  on having  one vertex  of maximal
  size} $\delta$.

\medskip

Now we consider  a Kesten  tree $\ck$ under $\P^\psi$.
We  define the
tree $\cknr$ as follows:  let $H'_\delta$ be the height
of the lower vertex on the spine  on  which is grafted a tree  with a vertex
(possibly  its root)  of size  larger or  equal than  $\delta$. Notice
that $H'_\delta$ is a.s.\ finite as $\pi([\delta,\infty ))>0$. 
Let $E$   be  under $\P^\psi$ an independent exponential random variable
with mean $1/ \psi'(0)$, with the convention that $E=+\infty $ in the
critical case, and set $H_\delta=\min (H'_\delta, E)$.  We also identify
$H_\delta$  with the vertex on the infinite spine of $\ck$ at distance
$H_\delta$ from the root. 
We define   $\cknr$ as the  closure of the connected
component containing the root of the  Kesten tree $\ck$ when the vertex
$H_\delta$ has been removed.  More formally, we have:
\begin{equation}
   \label{eq:def-cknr}
  \cknr
  = \II{0, H_\delta}  \circledast_{i\in
    I^{\node}_\delta}( \ct_i,h_i)
  \quad\text{with}\quad
  I^{\node}_\delta=\{i\in I\, \colon\, h_i <H_\delta\},
\end{equation}

Let $(h''_j,\ct''_j)_{j\in J}$ be the atoms of a Poisson point process on
$\R_+ \times \T$ with
intensity $\rd h \, \N^\psi[\rd \ct]$ and independent of the Kesten tree
$\ck$.  Recall  $\{\root\}$  denote   the trivial 
rooted  real tree reduced  to the  root.
We set for $\delta>0$:
\begin{equation}
   \label{eq:def-T[d]}
  \ct_{[\delta]}^\node=\{\root\}  \circledast_{j\in J^{\node}_\delta} (\ct''_j, \root)
\quad\text{with}\quad
J^{\node}_\delta=\{j\in J\, \colon\, h''_j\leq  \delta \quad\text{and}\quad
\Delta(\ct''_j)\leq  \delta \},
\end{equation}
so that $\ct_{[\delta]}^\node$ is independent  of $\ck$  and distributed  according to
$\P^\psi_\delta(  \cdot |\,  \Delta\leq \delta)$,  with $\P_\delta^\psi$
defined in Section~\ref{sec:meas-LT}.
We then denote by $\ckn$ the tree $\cknr$ on which the tree
$\ct_{[\delta]}^\node$ is  grafted at the top
$H_\delta$  of the  truncated  spine $\II{0,  H_\delta}$:
\begin{equation}
   \label{eq:def-ckn}
\ckn=  \cknr \circledast (\ct_{[\delta]}^\node,
H_\delta).
\end{equation}
Notice that $\wbr(\ckn, H_\delta)= \cknr$.
The next theorem is proved in Section~\ref{sec:proof-D}. 

\begin{theo}[Coupling for the  maximal size vertex conditioning]
  \label{theo:Delta}
  Assume~\eqref{eq:hyp-param}-\eqref{eq:critical}   hold and that  $\pi\neq 0$. Let  $\delta>0$ such that
  $\pi(\{\delta\})=0$ and $\bar \pi(\delta)>0$. 
Then,
  the Lévy tree  under $\N^\psi$ conditionned to have  maximal  vertex size
  $\delta$, $\ctn$,     is   distributed    as   the    truncated   Kesten    tree
  $\ckn$ under $\P^\psi$, more precisely, we have on $\T_1$:
\begin{equation}
   \label{eq:=en-loi-Td}
    (\ctn,X_\delta)\overset{\text{(d)}}{=}   (\ckn, H_\delta)
    \quad\text{and}\quad
    (\ctnr, X_\delta) \overset{\text{(d)}}{=}   (\cknr, H_\delta).
  \end{equation}  
\end{theo}
The  first equality
in~\eqref{eq:=en-loi-Td} is  a consequence of the second one and of~\eqref{eq:Tnode*}
and~\eqref{eq:def-ckn}.

\begin{rem}[Reconstruction of the L\'evy tree]
   \label{rem:recons-vertex}
 Let  $\ck$  be  an
independent   Kesten    tree.   Assume that
$\pi$   has    no   atom and 
$\pi(\R_+^*)=+\infty  $   (that  is,   $\N^\psi[\Delta=0]=0$). If 
 $\Delta$ is distributed as  $\N^\psi[\rd  \Delta(
\ct)]$  and
independent  of  $\ck$ and  of  the atoms $(h''_j,\ct''_j)_{j\in J}$ of the
independent  Poisson point process which appears in~\eqref{eq:def-T[d]},  we deduce  from
Theorem~\ref{theo:Delta} that $ \cknD$, defined as in~\eqref{eq:def-ckn}
with $\delta$ replaced by $\Delta$, is distributed as
the  unconditioned  Lévy  tree   $\ct$  under  $\N^\psi[\rd  \ct]$.  The
condition  $\pi(\R_+^*)=+\infty  $  can  be   removed  at  the  cost  of
considering  the event  $\{\Delta(\ct)=0\}$  which has  then a  positive
measure under $\N^\psi[\rd \ct]$.
\end{rem}

By construction, we have the following result on the coupled Kesten
sub-trees. Notice the increasing limit of $\ct^\node_{[\delta]}$ as $\delta$ goes
to infinity is not  locally compact as there is a condensation
phenomenon at the root; thus  the sequence
$(\ct^\node_{[\delta]})_{\delta>0}$  does not converge in $\T$. One would need to
enlarge the set $\T$ and change the topology for this sequence to
converge, see~\cite{Abraham2014a} in the discrete setting of
Bienaym\'e-Galton-Watson  trees. For this reason, we can only consider
local limit when $\lim_{\delta\rightarrow \infty } d^\ck(0,
H_\delta)=+\infty $, that is when the branching is critical.

\begin{prop}[Monotony and local convergence for the maximal size vertex coupling]
  \label{prop:delta}
  Assume~\eqref{eq:hyp-param}-\eqref{eq:critical}  hold  and  $\pi\neq 0$. If  $\delta'\geq \delta>0$ are
  not atoms of $\pi$ and  $\bar\pi(\delta')>0$, then we have a.s.\ that:
\begin{equation}
   \label{eq:mono-Kd}
  \cknr \subset
  \cknrp, \quad
  H_\delta\in \II{\root, H_{\delta'}}
    \quad\text{and}\quad
\ct^\node_{[\delta]}\subset \ct^\node_{[\delta']}.
  \end{equation}  
  Furthermore,  in  the  critical  case,   we have
  $\lim_{\delta\rightarrow \infty } d^\ck(0,
H_\delta)=+\infty $ and,  if the  support  of  $\pi$  is
  unbounded and $\pi $ has no atom,
  the following  convergences hold  in $(\T,  \dlgh)$:
\begin{equation}
   \label{eq:cv-Kd}
    \cknr \xrightarrow[\delta\rightarrow +\infty]{\text{a.s.}} \ck
   \quad\text{and}\quad
    \ckn \xrightarrow[\delta\rightarrow +\infty]{\text{a.s.}} \ck.
  \end{equation}  
\end{prop}

\begin{proof}
  Equation~\eqref{eq:mono-Kd}
  is obvious as by construction $H'_\delta\leq  H'_{\delta'} $
  and thus $H_\delta\leq  H_{\delta'}$. 

  We now consider the critical case. Notice        that
  $R_a(\ckn)=R_a(\ck)$ for all $a$ 
  strictly   less than       $H_\delta $.  Then use~\eqref{eq:rh(t)} and
  that a.s.\
  $\lim_{\delta\rightarrow +\infty } H_\delta=+\infty $ to get 
  the convergences in~\eqref{eq:cv-Kd}. 
\end{proof}

We  deduce from Theorem~\ref{theo:Delta}
and Proposition~\ref{prop:delta} the following  well known result
from~\cite{nassif2022}.

\begin{theo}[Local convergence for the maximal size vertex conditioning]
\label{theo:cv-loc-size-vertex}
  Assume~\eqref{eq:hyp-param}-\eqref{eq:grey}  hold,  the regime is critical  (that is $\psi'(0)=0$),
  the L\'evy measure   $\pi$  has no atom and its support is
  unbounded.  The following 
local   convergence holds in distribution:
 \[
   \ctn  \xrightarrow[\delta\rightarrow +\infty]{\text{(d)}} \ck
     \quad\text{in}\quad
  (\T, \dlgh).
\]
\end{theo}

\begin{rem}[The sub-critical case]
  \label{rem:sub-critic-d}
  Notice   that~\eqref{eq:cv-Kd}   no   longer  holds   if   $\psi$   is
  sub-critical. Assume the support of $\pi$  is unbounded and $\pi $ has
  no  atom.  According  to~\cite[Theorem~1.6]{nassif2022}  the limit  of
  $\ctn$ is informally a tree with a condensation vertex at level given by
  an  exponential   random  variable   with  mean   $1/\psi'(0)$.   From
  Theorem~\ref{theo:Delta},           we           deduce           that
  $\lim_{\delta\rightarrow   +\infty   }    H_\delta=E   $,   and   thus
  $  \cknr \xrightarrow[\delta\rightarrow  +\infty]{\text{a.s.}}
  \wck_{E}
  $, where  $\wck_{E}=\wbr(\ck, E)$ is  the Kesten  tree whose spine  has been  cut at
  level   $E$.   Now   informally   the   increasing   random   sequence
  $(\ct^\node_{[\delta]})_{ \delta>0}$ converges  to a  (non locally  compact) tree
  with   condensation at the  root; and it  should then be  possible to
  give  sense  to  the  convergence  of $  \ckn$  towards  a  tree  with
  condensation at vertex  $E$.
\end{rem}

\subsection{Proof of Theorem~\ref{theo:Delta}}
\label{sec:proof-D}
We assume~\eqref{eq:hyp-param}-\eqref{eq:critical} and let $\delta>0$ be
such    that     $\pi(\{\delta\})=0$    and     $\bar    \pi(\delta)>0$,
see~\eqref{eq:def-pid}.    Notice   the   function   $\psi_{\{\delta\}}$
from~\eqref{eq:def-psi-d}                                      satisfies
also~\eqref{eq:hyp-param}-\eqref{eq:critical} and it is then a bijection
of $\R_+$.   We define  $\phi$ as  $\psi_\theta$ in~\eqref{eq:def-psi-q}
with  $\psi$   replaced  by  $   \psi_{\{\delta\}}$  and $\theta$ by:
\begin{equation}
  \label{eq:ND}
\psi_{\{\delta\}}^{-1}(\bar \pi(\delta))=\N^\psi[\Delta>  \delta]=\N^\psi[\Delta\geq  \delta]>0,
\end{equation} 
(where we used that  $\delta$ is not an atom of $\pi$ for the second equality)
that is, for $\lambda\geq 0$:
\begin{equation}
   \label{eq:def-phi}
  \phi(\lambda)= 
   \psi_{\{\delta\}}(\lambda+\psi_{\{\delta\}}^{-1}(\bar \pi(\delta))) - 
\bar \pi(\delta).
\end{equation}
The
characteristic 
$(\alpha_\phi, \beta_\phi, \pi_\phi)$ of $\phi$ are given by:
\[
  \alpha_\phi  =\phi'(0)= \psi_{\{\delta\}}'(\N^\psi[\Delta \geq  \delta])> 0,\quad
  \beta_\phi  = \beta
  \quad\text{and}\quad
    \pi_\phi (\rd r)
  = \ind_{\{r< \delta\}}\, \expp{- r\N^\psi[\Delta\geq  \delta]} \, \pi (\rd r).
\]
In particular the
branching mechanism $\phi$ is sub-critical.
We deduce
from~\eqref{eq:Girsanov-sigma} that:
\begin{equation}
  \label{eq:Girsanov-T}
  \N^{\psi_{\{\delta\}}} \left[ \expp{-\bar \pi(\delta)  \sigma} F(\ct)\right]
  =  \N^{\phi} \left[  F(\ct)\right]. 
\end{equation}

Recall  we only need to prove the second equality
in~\eqref{eq:=en-loi-Td}.  Let  $F$ be a measurable nonnegative function
defined on $\T$. Using~\cite[Eq.~(4.2)]{nassif2022} for the first
equality and~\eqref{eq:Girsanov-T} for the second, we
have:
\begin{equation}
   \label{eq:(4.2)}
  \N^\psi\left[\ind_{\{\Delta \leq
  \delta\}} \, F(\ct)
  \right]
  =
  \N^{\psi_{\{\delta\}}}\left[\expp{- \bar\pi(\delta)  \sigma}
   \, F(\ct )
 \right]
 = \N^\phi[F(\ct)].
\end{equation}

The distribution of the marked random
tree $(\ctnr, X_\delta)\in \T_1$, with
$X_\delta$ the only vertex of size $\delta$ in $\ctn$
is characterized by~\eqref{eq:EGTd}. 
We have:
\[
  \N^\psi\left[\ind_{\{\Delta \leq
  \delta\}} \, \int_\ct \bm (\rd  x) \, G(\ct, x)
  \right] 
=  \N^{\phi}\left[ \int_\ct \bm (\rd  x) \, G(\ct, x)
     \right]
= \N^\phi[\sigma]\,  \E^{\phi}\left[  G(\ck_E, E)
    \right],
\]
where    we    used~\eqref{eq:(4.2)}    for    the    first    equality,
and   Corollary~\ref{cor:TX=Kh}  for  the  second  with, under $\P
^\phi$,   $E$  an
exponential  random  variable  with mean  $\N^\phi[\sigma]=1/  \phi'(0)$
independent of the Kesten tree $\ck$.

\medskip

Recall  the function  $\phi$ depends  on $\delta$.  Taking $G=1$  in the
previous equality gives:
\[
 \N^\psi \left[ \sigma \, \ind_{\{\Delta \leq
     \delta\}}  \right]
 =\N^\phi[\sigma].
 \]
So we deduce from~\eqref{eq:EGTd} that the proof of the second equality
in~\eqref{eq:=en-loi-Td} is complete, once we prove that:
\begin{equation}
   \label{eq:TBP}
  \E^\psi\left[G(\cknr, H_\delta)\right]
  =
   \,  \E^{\phi}\left[ G(\ck_{E}, {E})
\right].
\end{equation}

\medskip

Consider the Kesten tree $\ck$ under $\E^\psi$. The intensity of
grafting a tree having a vertex with size larger than $\delta$ is $\rd h
\,\bN^\psi[ \Delta\geq \delta]$ with  $\bN^\psi$
given in~\eqref{eq:def-bN}, so the height $H'_\delta$ at which is grafted a
tree with maximal vertex size larger than $\delta$   is an exponential random
variable with parameter:
\begin{align*}
  \bN^\psi[ \Delta\geq \delta]
  &= 2\beta \N^\psi[\Delta \geq \delta] + 
  \int_0^{+ \infty } r \pi(\rd r)\P_r^\psi (\Delta \geq  \delta)\\
&  =
  2  \beta \N^\psi[\Delta \geq \delta]
  +\int_{(0,\delta) } r \pi(\rd
  r) \left(1 - \expp{-r \N^\psi [\Delta\geq  \delta]}\right)
+
  \int_{[\delta, \infty )} r\pi(\rd r)\\
  &= \psi'_\delta(\N^\psi[\Delta \geq \delta]) - \psi'(0)\\
 &= \phi'(0) - \psi'(0),
\end{align*}
where we used~\eqref{eq:def-psi-d}  for the third 
equality with $\pi(\{\delta\})=0$, and~\eqref{eq:def-phi} 
 for the last.
We deduce that $\min(H'_\delta, E)$ with $E$ independent of $\ck$ and
distributed under $\P^\psi$ as an exponential random variable with mean
$1/\psi'(0)\in (0, \infty ]$ is an exponential random variable with mean
$1/\phi'(0)\in (0, \infty )$, which is the distribution of $E$ under
$\P^\phi$. 

Since  $\bar \pi(\delta)>0$, the random tree $\cknr$ conditionally on $H_\delta$
is distributed under $\P^\psi$ as the cut spine  $\II{0, H_\delta}$ on which are grafted trees:
\[
  \cknr
  \, \overset{\text{(d)}}{=} \,
  \II{0, H_\delta} \circledast_{i\in
    I'}( \ct'_i,h'_i),
\]
where  $(h'_i, \ct'_i)_{i\in  I'}$  are  the atoms  of  a Poisson  point
measure      on      $\R_+      \times      \T$      with      intensity
$\rd  h \,  \bN^\psi[\rd  \ct; \,  \Delta<\delta]$.
So to prove the second equality in distribution
of~\eqref{eq:=en-loi-Td}, we deduce from~\eqref{eq:TBP} that  it is enough to
check that the grafting intensity on the spine $\II{0, H_\delta}$ for $\cknr$,
that is $\bN^\psi[\rd \ct \, ;\, \Delta<\delta]$, coincides with 
 the grafting intensity on the spine $\II{0, E}$ for $\ck$ under $\P^\phi$,
that is, with $\bN^\phi[\rd \ct ]$.

Recall  $\P^\phi_r$   (with  $\psi$   instead  of  $\phi$)   defined  in
Section~\ref{sec:meas-LT} denotes  the distribution  of the  random tree
$\ct_{[r]}=\{\root\}   \circledast_{i\in   I''_r}(    \ct''_i,   \root)$,   where
$(h''_i, \ct''_i)_{i\in I''}$ are the atoms of a Poisson point measure
on $\R_+ \times \T$ with intensity $\rd h \, \bN^\phi[\rd \ct]$
and $I''_r=\{i\in I''\, \colon\, h''_i\leq r\}$. 
We first notice that  for $G$ a measurable nonnegative function on $\R_+
\times \T$ and $r>0$:
  \begin{align*}
  \E_r^\phi\left[\expp{-\sum_i G(h_i, \ct_i)}\right]
  & = \exp\left\{- \int_0^r \rd s\,  \N^\phi [1 - \expp{-G(s,
    \ct)}]\right\}\\
  &= \exp\left\{- \int_0^r \rd s\,  \N^\psi [(1 - \expp{-G(s,
    \ct)})\ind_{\{\Delta< \delta\}}]\right\}\\
  &=  \exp\left\{- \int_0^r \rd s\,  \N^\psi [1 - \expp{-G(s,
    \ct)}\ind_{\{\Delta< \delta\}}] +r \N^\psi
    [\Delta\geq \delta]\right\}\\
    &= \expp{r \N^\psi
    [\Delta\geq \delta]} \E^\psi_r \left[\expp{-\sum_i G(h_i,
      \ct_i)\ind_{\{\Delta(\ct^i) < \delta\}}}\right],
  \end{align*}
  where we used~\eqref{eq:(4.2)} for the second equality.
Arguing similarly as in the proof of~\eqref{eq:Ph'=ecPh}
given by the computations in~\eqref{eq:compute!}, we deduce 
 that for $r\in (0,\delta)$:
\[
  \P ^\psi_r
    (\rd \ct;  \,  \Delta<\delta)= 
 \expp{-r \N^\psi[\Delta\geq \delta]}  \P^\phi_r
 (\rd \ct).
\]
Then, using~\eqref{eq:(4.2)}, we obtain:
\begin{align*}
  \bN^\psi[\rd  \ct; \,  \Delta<\delta]
  &= 2 \beta \N^\psi[\rd \ct; \,  \Delta<\delta] + \int_{(0,
    +\infty )} r\pi(\rd r) \, \P ^\psi_r
    (\rd \ct;  \,  \Delta<\delta)\\
  &= 2 \beta \N^\phi[\rd \ct] + \int_{(0, \delta)} r\pi(\rd r)\,
    \expp{-r \N^\psi[\Delta\geq \delta]}  \P^\phi_r
    (\rd \ct)\\
  &= 2 \beta \N^\phi[\rd \ct] + \int_{(0, \delta)} r\pi_\phi(\rd r)\,
    \P^\phi_r    (\rd \ct)\\
  &=  \bN^\phi[\rd  \ct].
\end{align*}
This concludes the proof of the  second equality
in~\eqref{eq:=en-loi-Td}.
(We already noticed that the first equality in~\eqref{eq:=en-loi-Td} is a
direct consequence of the second one.)

\section{Lévy tree with a  given mass}

\subsection{Subordinator and Bernstein function}
\label{sec:bernstein}

Let $\P^\varphi_x$ denote the  distribution of a 
  subordinator $\bS=(S_t)_{t\geq 0}$ starting at  $x\geq 0$ at time $0$, that
  is $S_0=x$ a.e., with general  Laplace     exponent $\varphi$: 
\begin{equation}
   \label{eq:varphi=}
  \varphi(\lambda)=  \kappa   \lambda + \int  _{(0,  +\infty   )}  (1
  -\expp{-\lambda r})\,  \nu(\rd r) \quad \text{for $\lambda\geq 0$},
\end{equation}
  where the drift $ \kappa$ belongs to $ \R_+$ and the L\'evy measure 
$\nu$    is    a   measure    on    $(0,    +\infty   )$    such    that
$\int_{(0, +\infty )}  (1 \wedge r) \, \nu(\rd r)$  is finite. In
particular, we have  for $\lambda, t\geq 0$ that:
\[
  \E^\varphi_x[\expp{-\lambda S_t}]= \expp{- t
  \varphi(\lambda) - \lambda x}. 
\]
We recall that:
\begin{equation}
   \label{eq:kappa=phi/lambda}
 \lim_{\lambda \rightarrow \infty }
 \frac{\varphi(\lambda)}{\lambda}= \kappa
 .
\end{equation}
The
corresponding  $\alpha$-potential measure  $\pot^{(\alpha)}$  on $\R_+$,
for $\alpha\geq 0$, is defined  by
$\pot^{(\alpha)}(A)=\int_0^\infty \rd t \, \expp{-\alpha t} \P(S_t\in A)$ for all
Borel sets $A\subset \R_+$. Its Laplace transform is given by:
\begin{equation}
   \label{eq:Lap-pot}
  \int_{[0, \infty )} \expp{-\lambda r}\, \pot^{(\alpha)} (\rd r)
  = \inv{\alpha +\varphi(\lambda)}
  \quad\text{for all $\lambda\geq 0$.}
\end{equation}
We simply write $\pot$ for the $0$-potential $\pot^{(0)}$.

\begin{rem}
   \label{rem:Bernstein}
   We recall that a Bernstein  function, say $\varphi_0$, is the Laplace
   exponent   of  a   killed  subordinator   (that  is,   of  the   form
   $\varphi_0=\varphi+c_0$, with $c_0\geq 0$)  and that $\varphi_0$ is a
   special  Bernstein  function  if   the  function  $\phi$  defined  by
   $\phi(\lambda)=\lambda/  \varphi_0(\lambda)$  is   also  a  Bernstein
   function, see~\cite[Sec.~11]{ssrv:bernstein}.

   Provided that $\bS\not  \equiv 0$, the Laplace  exponent $\varphi$ of
   the subordinator $\bS$ is a special Bernstein function if and only if
   the $0$-potential measure $\pot(\rd r)$  of $ \varphi$ can be written
   as   $\mathrm{c}  \delta_0   (\rd  r)   +  u(r)   \,  \rd   r$,  with
   $\mathrm{c} \geq  0$ (and  $\delta_0$ the  Dirac mass  at 0)  and the
   nonnegative   density   function   $u$   defined   on   $\R_+^*$   is
   non-increasing.          Thus          we          have          that
   $u(\infty )=  1/(\kappa + \int  _{(0, \infty  )} r \nu(\rd  r))$, and
   also,   if  $\kappa>0$   or  $\nu(\R_+^*)=+\infty   $,  that   $c=0$,
   $\lim_{r\rightarrow         0+}         r         u(r)=0$         and
   $\lim_{r\rightarrow 0+} \int_0^r u(s) \, \rd s=0$.
\end{rem}

\subsection{The subordinator conditioned to die at a given level}
\label{sec:sub-die-hard}
In this section, we  follow \cite{krs:csemp}.  Let $\bS=(S_t)_{t\geq 0}$
be     a    subordinator     with     Laplace     exponent given
by~\eqref{eq:varphi=}. 
 We assume
that    $\varphi(+\infty   )=+\infty    $,    that    is   $\kappa>0$    or
$\nu((0, 1))=+\infty  $.
Notice  that the $0$-potential $\pot=\pot^{(0)}$ is well defined
as for any $\lambda>0$:
\[
  \pot([0, x])=\int_0^\infty \rd t \, \P^\varphi(S_t \leq  x) \leq  \int_0^\infty
  \rd t \, \E^\varphi[\expp{\lambda (x-S_t)}] =\inv{\varphi(\lambda)} \expp{\lambda
    x}< +\infty .
\]
 Notice also  that $\varphi(+\infty  )=+\infty $
implies  that $\P^\varphi(S_t=0)=0$ for all $t>0$ and thus
$\P^\varphi(S_t=a)=0$ for all $a>0$, see~\cite{b:levy} p.~30.  
For   $\alpha\geq   0$,   we   deduce   that   the   potential   measure
$\pot^{(\alpha)}$ has no  atom.  We recall that  if $\pot^{(\alpha)}$ is
absolutely  continuous  w.r.t.\  Lebesgue  measure on  $\R_+$  for  some
$\alpha\geq  0$,  then  it  is absolutely  continuous  w.r.t.\  Lebesgue
measure on $\R_+$ for any $\alpha\geq 0$.

\medskip

The results from~\cite{krs:csemp} are stated for $\alpha=0$, but they
can be extended to the case $\alpha\geq 0$ with minor modifications of
the proofs which we shall omit. 
We   now   fix   $\alpha\geq   0$  and   assume   that   the   potential
$\pot^{(\alpha)}$ has  a continuous  density $u^{(\alpha)}$  w.r.t.\ the
Lebesgue measure on $(0, \infty )$, see~\cite[Hyp. (DA)]{krs:csemp} when
$\alpha=0$.

Let $r>0$. Following Lemma~2.4  therein,  the  process $(M^r_t)_{t\geq
  0}$ defined by:
\begin{equation}
   \label{eq:def-super-mart}
M_t^r= \expp{-\alpha t} u^{(\alpha)} (r   -S_t)   \ind_{\{
  S_t<r\}}
\end{equation}
is   a 
super-martingale. 
Using that the support of the distribution of $S_t$ is $\R_+$ for $t>0$,
 see~\cite{tucker:75}, it  is  then  easy  to  check  that  the  density
$u^{(\alpha)}$  is   positive  on  $(0,  +\infty   )$.   
So,  we can define  the following h-transform  of the
Lévy  process  $\bS$   with  lifetime  $\zeta^r$    and
starting at $x\in [0, r)$ under $\P^{\varphi, r}_x$
defined   by,  for  all
$t\in  [0, \infty  )$,  $A \in  \cf_t=\sigma((S_s)_{s\in  [0, t]})$:
\begin{equation}
   \label{eq:cond-S-die}
  \P^{\varphi,r}_x(A, t<\zeta^r)=\expp{-\alpha t}  \E^\varphi_x\left[
\frac{u^{(\alpha)}(r-S_t)}{u^{(\alpha)}(r-x)} \ind_{\{ S_t <  r\}}\, \ind_ A
\right].
\end{equation}

Since $\P^\varphi_x$-a.s.\ $\lim_{t\rightarrow \infty } S_t=+\infty$, we get that $\lim_{t\rightarrow
  \infty } \P^{\varphi, r}_x (t< \zeta^r)=0$, and thus  the lifetime $\zeta^r$ is a.s.\ finite. 
Under $\P_x^{\varphi, r}$, the process $\bS$ is started at $x$, killed at
time $\zeta^r$ and 
continuously absorbed at $r$, that is $\P_x^{\varphi, r}$ a.s.\ $S_{\zeta^r-}=r$.
Indeed, for $y\geq 0$, set $\tau_y=\inf\{t\geq 0 \, \colon\, S_t >y\}$.
Extending  Theorem~2.5 in~\cite{krs:csemp} given for $\alpha=0$ to any $\alpha\geq 0$, we also get that for $0\leq  x < b<
r$, $t\geq 0$  and $A\in \cf_t$:
\begin{equation}
  \label{eq:P-phi,r_x}
  \P^{\varphi, r}_x(A, t<\tau_b)=
  \lim_{\varepsilon \downarrow 0 }
  \frac{
    \E^\varphi_x\left[  \expp{-\alpha \tau_r}; \, A, t< \tau_b ,  S_{\tau_r-} \geq  r-\varepsilon
     \right]}
  {
    \E^\varphi_x\left[  \expp{-\alpha \tau_r}; \,   S_{\tau_r-} \geq  r-\varepsilon\right]}
  \cdot
\end{equation}
For $\alpha=0$, see~\cite{krs:csemp}, this reduces to:
\[
  \P^{\varphi, r}_x(A, t<\tau_b)=
  \lim_{\varepsilon \downarrow 0 }
  \P^\varphi_x( A, t< \tau_b \, |\, S_{\tau_r-} \geq  r-\varepsilon).
\]
Eventually, we simply write $\P^{\varphi,r}$ for $\P^{\varphi,r}_x$
when $x=0$.

\subsection{On the density of $\sigma$ under the excursion measure}
\label{sec:density-sigma}

We assume~\eqref{eq:hyp-param}-\eqref{eq:critical}  and thus $\psi'(0)\geq 0$, which corresponds to
the  (sub-)critical case. 
We shall consider the associated critical branching mechanism $\psicr$
defined by:
\[
  \psicr(\lambda)=\psi(\lambda)- \psi'(0)\lambda.
\]
Notice that $\psi'$ and $\psicr'$ are Bernstein functions as for $\lambda \geq 0$:
\[
  \psi'(\lambda)=\psi'(0) +
  \psicr'(\lambda)
 \quad\text{and}\quad
 \psicr'(\lambda)=
  2\beta \lambda + \int_{(0, \infty )}
  \left(1-\expp{-\lambda r}\right) \, r\pi(\rd r).
\]

Let $\bV=(V_t)_{t\geq 0}$ be a  subordinator  with Laplace
exponent $\psicr'$.
Let $\bT=(T_t)_{t\geq 0}$ be a subordinator independent
of $\bV$ with no killing term, no drift and Lévy measure $\N[\rd \sigma]$, that is,
with Laplace exponent $\psi^{-1}$ by~\eqref{eq:N[1-S]}.
 The process $\bS=(S_t(\omega)=T_{V_t(\omega)}(\omega))_
{t\geq 0}$ is thus a subordinator with Laplace exponent:
\begin{equation}
   \label{eq:def-phi2}
  \varphi= \psicr' \circ \psi^{-1}.
\end{equation}

\begin{rem}
  \label{rem:kappa=0}
  Since $\varphi$ is the Laplace exponent of the subordinator $\bS$, it can be
  written as in~\eqref{eq:varphi=}. Notice
  that~\eqref{eq:hyp-param}-\eqref{eq:critical} imply that $\psi(+\infty
  )=\psicr'(+\infty )=\varphi(+\infty )=+\infty $. 
Using that $x(1-\expp{-x}) \geq \expp{-x}-1+x \geq  x(1-\expp{-x})/2$
for $x\geq 0$,
we get that for $\lambda\geq 0$:
\begin{equation}
   \label{eq:psi-psi'}
\psi(\lambda) \leq  \lambda \psi'(\lambda) \leq  2 \psi(\lambda).
\end{equation}
This implies that:
\begin{equation}
   \label{eq:cv-phi-psi}
  \lim_{\lambda\rightarrow \infty } \frac{\varphi(\lambda)}{\lambda}
  = \lim_{\lambda\rightarrow \infty } \frac{\psi'(\lambda)}{\psi(\lambda)}=0,
\end{equation}
and thus $\kappa=0$ by~\eqref{eq:kappa=phi/lambda}. We deduce 
that $\nu((0, 1))=\infty $ and 
from~\cite{b:levy} p.~30 that $\P(S_t=0)=0$ for all $t>0$ and thus
$\P(S_t=a)=0$ for all $a>0$.  In particular, we get that the
corresponding $\alpha$-potential has no atom on $\R_+$. 
\end{rem}

We denote by $\pot^{(\alpha)}$ the $\alpha$-potential associated to
$\bS$. 

\begin{lem}
  \label{lem:U=N}
  Under~\eqref{eq:hyp-param}-\eqref{eq:critical}, with
  $\alpha=\psi'(0)\geq 0$, we have for any  Borel set
  $A  \subset \R_+$ that:
  \[
\pot^{(\alpha)}(A)=\N^\psi[\sigma \ind_A(\sigma)].
\]
In particular the      measure
  $\N^\psi[\rd \sigma]$ has  no atom on $[0, \infty )$.
\end{lem}
\begin{proof}
   Using that $\N^\psi[1- \expp{-\lambda
  \sigma}]=\psi^{-1}(\lambda)$, we get that for $\lambda>0$:
\begin{equation}
   \label{eq:U-sigma}
  \int _{[0, \infty )} \expp{-\lambda r}\, \pot^{(\alpha)}(\rd r)= \int_0^\infty
    \expp{-\alpha t} \, \E\left[\expp{-\lambda S_t}\right] \, \rd t= \inv{ \psi'\circ
      \psi^{-1} (\lambda)} = \partial_\lambda \N^\psi[1 - \expp{-\lambda
      \sigma}]=\N^\psi\left[\sigma \, \expp{-\lambda \sigma}\right],
  \end{equation}
  and thus  $\pot^{(\alpha)} (A)=\N^\psi[\sigma \ind_A(\sigma)]$  for any  Borel set
  $A  \subset \R_+$.
By  Remark~\ref{rem:kappa=0}, 
 the measure $\pot^{(\alpha)}$ has no atom
  on   $\R_+$,   and   thus,    as  $\N^\psi[\sigma=0]=0$ by~\eqref{eq:N[1-S]},    the
  excursion   measure
  $\N^\psi[\rd \sigma]$ has  no atom on $[0, \infty )$. 
\end{proof}

 Notice 
  from~\eqref{eq:U-sigma} that  $\sigma$ under  $\N^\psi$ has  a density
  w.r.t.\  the  Lebesgue measure  on $[0,  \infty )$ if  and only  if the
  $\alpha$-potential $\pot^{(\alpha)}$ of  $\bS $ has  a density  w.r.t.\ the Lebesgue  measure on
  $(0, \infty )$.

\begin{rem}[The critical stable case]
  \label{rem:stable}
 We consider the 
critical stable case $\psi(\lambda)=\lambda^a$ with $a\in (1,
2]$ (and thus $\psi'(0)=\alpha=0$). From~\eqref{eq:U-sigma}, we deduce that, with $u=u^{(0)}$ the
density of $\pot^{(0)}$ on $(0, +\infty )$:
\[
  f_\sigma(r)= \inv{a  \Gamma\left(\frac{a-1}{a} \right)\, r^{(a+1)/a}}
  \quad\text{and}\quad
  \frac{u(r-x)}{u(r)}=  \left( \frac{r}{r-x}\right)^{1/a}
  \quad\text{for}\quad r>x>0.  
\]
\end{rem}

\subsection{The Lévy tree  conditioned  by the total mass}
\label{sec:Levy-cond-mass}

We  assume~\eqref{eq:hyp-param}-\eqref{eq:critical}, and we refer to
Section~\ref{sec:mass-super} for the super-critical case.   For  simplicity we
write $\sigma$ for  the total mass $ \bm(\ct)$ of  the L\'evy tree $\ct$
(see  Section~\ref{sec:LT}),  and the  tree  $\ct$  which is  implicitly
considered   in   $\sigma$   shall    be   clear   from   the   context.
By~\cite[Corollary~1.29]{k:rm}, there  exists a  regular version  of the
conditional   distribution    of   $\ct$    given   its    total   mass,
$\N^\psi[\rd \ct  \, |\, \sigma=r]$ for  $r>0$, and let $\ctm$  denote a
$\T$-valued random  variable with  this distribution.  In what follows we shall  give a
nice representation of $\ctm$ and of $\N^\psi[\rd \ct \, |\, \sigma=r]$.

\begin{rem}[The stable case]
  \label{rem:stable-height}
  Assume  that $\psi(\lambda)= \lambda^a$ with $a\in (1,
  2]$. It is possible to give an explicit construction of
  $\ctm$, see~\cite[Sec.~3.2]{Duquesne2005b}. In particular, we
  have the following scaling property for $r>0$:
  \begin{equation}
      \label{eq:scaling-alpha}
    (\ctm,  d, \root)
  \, \overset{\text{(d)}}{=} \,
    (\ctmun,  r^{1-1/a} d,  \root).
 \end{equation}
 \end{rem}

  We can also represent the Kesten tree as:
\begin{equation}
   \label{eq:rep-K-sub}
  \ck= \spine  \circledast_{i\in
    I}( \ct^{\mass}_{\sigma_i},h_i),
\end{equation}
where $(h_i,  \sigma_i, \ct^{\mass}_{\sigma_i})_{i\in I}$ are  the atoms
of   a    Poisson   point    measure   on   $\R_+^2$    with   intensity
$\rd h \,  \bN^\psi[ \sigma\in \rd r]\,  \bN^\psi[\rd \ct|\sigma=r]$.  The
process $\bS^\ck=(S^\ck_h)_{h>0}$ defined by:
\begin{equation}
   \label{eq:rep-S}
  S^\ck_h=\sum_{h_i\leq  h}  \sigma_i  ,
\end{equation}
is, thanks to~\eqref{eq:bN-moment},
a  subordinator with Laplace exponent:
\begin{equation}
  \label{eq:def-phi3}
  \varphi=\psicr'\circ
\psi^{-1}.
\end{equation}
When  there   is  no   ambiguity,  we  shall   simply  write   $\bS$  for
$\bS^\ck$. Notice that $\varphi(0)= 0$.

\medskip Notice that the distribution of $\sigma$ under $\N^\psi$ has no atom
by  Lemma~\ref{lem:U=N} and  $\N[\sigma=+\infty  ]=0$  as the  branching
mechanism is (sub-)critical.  We further assume in this section that the
distribution  of the  total  mass  $\sigma$ on  $(0,  +\infty  )$ has  a
continuous density, say $f_\sigma$, w.r.t.\ the Lebesgue measure:
\begin{equation}
   \label{eq:densite}\tag{H5}
   \boxed{\N^\psi[ \sigma\in\rd r]= f_\sigma(r)\, \rd r
     \quad\text{and}\quad
\text{$f_\sigma$ is continous on $\R_+^*$}.}
\end{equation}
In         particular,        Assumption~\eqref{eq:densite}         (and
also~\eqref{eq:decroissant} below)  is satisfied in the  critical stable
case    $\psi(\lambda)=\lambda^a$    with    $a\in    (1,    2]$,    see
Remark~\ref{rem:stable}.

\medskip

Under~\eqref{eq:densite}, we deduce from Remark~\ref{rem:kappa=0} and Lemma~\ref{lem:U=N} that the
$\alpha$-potential $\pot^{(\alpha)}$ of $\bS $  has no Dirac mass at $0$
and  has a  continuous density,  say  $u^{(\alpha)}$,  w.r.t.\ the  Lebesgue
measure      on     $(0,      \infty      )$;      thus     we      have
$\pot^{(\alpha)}(\rd    r)=u^{(\alpha)}(r)\,    \rd    r$    and    from
Lemma~\ref{lem:U=N} that:
\begin{equation}
  \label{eq:r*pi=u}
    r f_\sigma(r)= u^{(\alpha)}(r)
  \quad\text{for all}\quad
  r\in (0, +\infty ).
\end{equation}

\medskip

We define  a modified Kesten tree  $\ckm$ as follows. Let  $r>0$ and let
$\bS^r$    be    with     distribution    $\P^{\varphi,    r}_0$    from
Section~\ref{sec:sub-die-hard},  that  is,  when $\alpha=0$  defined  as
$\bS$    be     continuously    absorbed    at    $r$.      Denote    by
$(h'_i,  \sigma'_i)_{i\in  I'}$  the  jumping times  and  the  jumps  of
$\bS^r$, and $\zeta^r$  the finite lifetime of $\bS^r$.  We consider the
$\T$-valued random variable defined by:
\begin{equation}
   \label{eq:def-ckm}
  \ckm= \II{0, \zeta^r}   \circledast_{i\in
    I'}( \ct^{\mass}_{\sigma'_i},h'_i),
\end{equation}
where        the        compact        rooted        random        trees
$(  \ct^{\mass}_{\sigma'_i})_{i\in  I'}$  are conditionally  on  $\bS^r$
independent  with  $\ct^{\mass}_{\sigma'_i}$  distributed  according  to
$\bN^\psi[\rd \ct |  \sigma=\sigma'_i]$ for all $i\in  I'$. In particular,
we see $\zeta^r$ as  a leaf of $\ckm$, and $(\ckm,  \zeta^r)$  as
an element  of $\T_1$.  We  shall denote  by $\P^{\varphi,r}$ the  distribution of
$(\ckm, \zeta^r)$.

 For $h< \zeta ^r$,
we shall  see $h$ as the  vertex of the branch  $\II{\root, \zeta^r}$ at
distance $h$ from  the root, as well as the  vertex of the semi-infinite
spine  of the  Kesten tree  $\ck$ at  distance $h$  from the  root.  From  the
representation~\eqref{eq:rep-K-sub} of the subordinator $\bS^\ck$ from the
Kesten tree $\ck$, notice that $S^\ck_h$ is the total mass
of the random tree $\wbr(\ck,h)$, the Kesten tree with the spine
truncated at level $h$, see~\eqref{eq:def-wbr}; it is a.s. larger than the total
mass  of  $\br_h(\ck)$,  the  Kesten  tree  truncated  at level  $h$;  the
difference comes from the fact that  $S^\ck_h$ counts also the mass above $h$
of the sub-trees grafted on the semi-infinite spine below level $h$.

\begin{rem}[On the distributions of $\ck$ and $\ckm$]
   \label{rem:version-reg}
We recall that for $(\bt, x)\in \T_1$, a tree $\bt$ with a distinguished
vertex $x\in \bt$, one can also see $x$ as a distinguished vertex of the
truncated tree
$\wbr(\bt, x)$, which is then considered as an element of $ \T_1$. 
Using~\eqref{eq:cond-S-die}, for any nonnegative measurable  function $G$ defined
  on   the  set   $\T_1$,  we   have,   with  $\bS^\ck$ given
  by~\eqref{eq:rep-S} and     the
  representation~\eqref{eq:rep-K-sub} of  the Kesten tree, that, with
  $\alpha=\psi'(0)\geq 0$,    for all
  $h>0$:
 \begin{equation}
   \label{eq:K=Sr+graft}
  \E^ {\varphi, r} \left[ \ind_{\{ \zeta^r > h\}}\, G(\wbr( \ckm, h) )
  \right] \, =\,   \expp{-\alpha h}
 \E^\psi\left[\frac{u^{(\alpha)}(r-S^\ck_h)}{u^{(\alpha)}(r)} \, \ind_{\{ S^\ck_h < r\}}\,
    G(\wbr( \ck, h))
     \right] .
  \end{equation} 
Notice that the right hand-side is a measurable function of $r\in (0, +\infty
)$ which is lower semi-continuous as the non-decreasing limits as $n\to +\infty$ of the
continuous functions $r \mapsto   \expp{-\alpha h}
 \E^\psi\left[\frac{n \wedge u^{(\alpha)}(r-S^\ck_h)}{u^{(\alpha)}(r)} \, \ind_{\{ S^\ck_h < r\}}\,
    G(\wbr( \ck, h))
     \right] $ for $n\in \N$ (the continuity of the latter functions is a consequence
     of the continuity of $u^{(\alpha)}$ and the fact that
     $\P^\psi$-a.s.\ $S^\ck_h\neq r$). 
\end{rem}

The proof of the next result is given in
Section~\ref{sec:proof-theo-mass}. 

\begin{theo}[A first representation for the mass conditioning]
  \label{theo:mass}
  Assume~\eqref{eq:hyp-param}-\eqref{eq:densite}  hold. For  $r\in (0,  +\infty )$, let $\ctm$ be the Lévy tree
  under $\N^\psi$  conditioned to have total  mass $r$, $U^r$ a  leaf of
  $\ctm$  chosen uniformly  (w.r.t.\  the probability  measure
  $r^{-1}  \bm^{\ctm}(\rd x)$)  and  on $\{\bH(U^r)\geq  h\}$ denote  by
  $U^r_h$ the  vertex of  $\II{\root, U^r}$ at  distance $h$  from the
  root of $\ctm$.   For any nonnegative measurable  function $G$ defined
  on   the  set   $\T_1$,  we   have,   with  $\bS^\ck$ given
  by~\eqref{eq:rep-S} and    the
  representation~\eqref{eq:rep-K-sub} of  the Kesten tree, that  $\rd
  r$-a.e., for all
  $h>0$:
 \begin{align}
   \label{eq:K=Tmass}
  \N^\psi_r\left[  \ind_{\{\bH(U^r)>h\}}\, 
   G( \wbr(\ctm, U_h^r)) \right]
       &\,   =\,   \E^ {\varphi, r} \left[ \ind_{\{ \zeta^r > h\}}\, G(\wbr( \ckm, h) )
  \right] .
  \end{align} 
\end{theo}

For  an  other local  absolute  continuity
representation, we also refer to
\cite[Proposition~3]{clg:bp} in  the critical quadratic case
($\psi(\lambda)=\beta \lambda^2$)
using the contour  process of the L\'evy tree (called  in this quadratic
case the  continuum Brownian tree)  given by the Brownian  excursion and
the  contour process  of  the  corresponding Kesten  tree  given by  two
independent 3-dimensional Bessel processes, and to 
\cite[Theorem~4]{cub:bridges-11} in the stable non-quadratic case
($\psi(\lambda)=\lambda^a$ with $a\in (1,2)$) with a representation of
the subordinator killed continuously at level $r$ (that is, under
$\P^{\varphi, r}_0$) using scaling of
bridges.

\medskip

We  deduce  the  following  corollary  using  first  the  regularity  of
$r\mapsto  \E^ {\varphi,  r} \left[  \ind_{\{ \zeta^r  > h\}}\,  G(\wbr(
  \ckm, h)  ) \right] $,  see Remark~\ref{rem:version-reg} and  that the
distribution  of   $(  \ctm,  U^r)$  under   $\N^\psi_r$  is  completely
determined               by                the               functionals
$ \N^\psi_r\left[ \ind_{\{\bH(U^r)>h\}}\, G( \wbr(\ctm, U_h^r)) \right]$
when $h$ ranges  over $\R_+$ and $G$ ranges over  the set of nonnegative
bounded measurable functions defined on $\T_1$.

\begin{cor}[Representation for the mass conditioning]
  \label{cor:mas}
  Assume~\eqref{eq:hyp-param}-\eqref{eq:densite}  hold. Let  $r\in (0,  +\infty )$.
The distribution of 
$(\ck^{\mass}_{r}, \zeta^r)$ under $\P^{\varphi, r}$ is a regular
version of the distribution of $(   \ctm, U^r)$, where 
 the random tree  $\ctm$ is distributed  as the  Lévy tree
  under  $\N^\psi$ conditioned  to  have mass  $r$, and 
  $U^r$ is  a leaf  of  $\ctm$  chosen uniformly  (that is, w.r.t.\ the probability
  measure $r^{-1} \bm^{\ctm}(\rd x)$). 
\end{cor}

\begin{rem}[Scaling properties of $\ckm$]
    \label{rem:scaling}
In the stable critical case $\psi(\lambda)=\lambda^a$ with $a\in (1, 2]$, a regular version of $\N_r^\psi$ can be obtained by scaling using~\eqref{eq:scaling-alpha}, see Remark~\ref{rem:stable-height}.
Let us check it also coincides with the one given by~\eqref{eq:K=Tmass}.

First, 
we have  $\varphi(\lambda) =a \lambda^{1- 1/a}$. Thus,  the subordinator $\bS$ enjoy the scaling property $r^{-1} S_\bullet \stackrel{(d)}{=} S_{(r^{-1+1/a}\bullet)}$. 
By~\eqref{eq:cond-S-die} with $\alpha=0$, the process $\bS^r$ started from 0 enjoy a scaling property, see the formula for the potential densities ratio in Remark~\ref{rem:stable}: $(S^r_t, t \in [0, \zeta_r)) \stackrel{d}{=} (rS^1_{r^{-1+1/a} \,t}, t\in [0, r^{1- 1/a}\,\zeta_1))$. In particular $\zeta^r$ and $r^{(1-1/a)}\,  \zeta^1$ have the same distribution. Furthermore, using~\eqref{eq:scaling-alpha}, the grafted trees in $\ckm$ enjoy also a similar scaling property (notice that the masses $\sigma'_j$ in $\ckm$ are distributed as the masses $r \sigma'_j$ in $\ckmun$). 
Then using the definition of $\ckm$ from~\eqref{eq:def-ckm}, we deduce that $(\ckm, d, \root)$ is distributed as $(\ckmun, r^{1-1/a} d, \root)$, which is in agreement with~\eqref{eq:scaling-alpha}. 
So the right-hand side of~\eqref{eq:K=Tmass} enjoys the same scaling property in $r$ as the left hand-side. 
This implies that the regular version of the conditional distribution of $\ctm$ given by scaling and the one given by~\eqref{eq:K=Tmass} coincide.  

\end{rem}

\subsection{Convergence for the conditioning by the total
  mass in the critical case}
We assume in this section that~\eqref{eq:hyp-param}-\eqref{eq:grey} and~\eqref{eq:densite} hold, that the regime is critical
(that is, $\alpha=\psi'(0)=0$) as well as a regularity condition on the
density $f_\sigma$ of $\sigma$ under $\N^\psi$:
\begin{equation}
   \label{eq:decroissant}\tag{H6}
   \boxed{\psi'(0)=0 \quad \text{and the function  $r\mapsto rf_\sigma(r)$ is non-increasing
on $(0, \infty )$.}}
\end{equation}

We simply write  $u$ for $u^{(0)}$, the density of  the potential of the
subordinator  with  Laplace  exponent $\varphi=  \psi'\circ  \psi^{-1}$.

\begin{rem}[On Assumption~\eqref{eq:decroissant}]
   \label{rem:u-Bernstein}
   Recall         we         assume~\eqref{eq:hyp-param}-\eqref{eq:grey}
   and~\eqref{eq:densite}.                   Notice                 that
   Assumption~\eqref{eq:decroissant}       is      equivalent,       by~
   \eqref{eq:r*pi=u},  to   the  function   $u$  being   continuous  and
   non-increasing on $(0, \infty )$.  By Remark~\ref{rem:Bernstein}, the
   function $u$  (or equivalently the function  $r\mapsto rf_\sigma(r)$)
   is  then  non-increasing  if  and  only if  $\varphi$  is  a  special
   Bernstein function, that is $\lambda/\varphi(\lambda)$ is the Laplace
   exponent of a subordinator.  (Notice  this subordinator is not killed
   as
   $\lim_{\lambda\rightarrow                                         0+}
   \lambda/\varphi(\lambda)=\lim_{\lambda\rightarrow         \emptyset+}
   \psi(\lambda)/\psi'(\lambda)=0$ by~\eqref{eq:psi-psi'}.)

  Assumption~\eqref{eq:decroissant} holds  in
   particular   in  the   critical  stable  case,   see  also
   Remark~\ref{rem:stable}.
\end{rem}

We now state the main result of this section. 

\begin{prop}[Strong local convergence for the mass conditioning]
  \label{prop:mass}
  Assume~\eqref{eq:hyp-param}-\eqref{eq:decroissant} hold (in particular the branching meachnism $\psi$  is
  critical, that is $\psi'(0)=0$). Let $F$ be bounded nonnegative measurable
  function on $\T$. We have:
 \begin{equation}
   \label{eq:main-ctm}
\lim_{r\rightarrow \infty } \N^\psi_r\left[  
 F( \br_h(\ctm)) \right]  
=\lim_{r\rightarrow \infty } \E^{\varphi, r}\left[  
 F( \br_h(\ckm)) \right]  
=\E^\psi\left[  F(\br_h( \ck)) \right].
\end{equation}
\end{prop}

\begin{proof}
 By~\eqref{eq:decroissant}, we get that the $0$-potential density $u$ exists and is
 non-increasing, so we get:
  \begin{equation}
   \label{eq:ineg-u/u}
   \frac{u(r-S^ \ck_h)}{u(r)} \, \ind_{\{ S_h^\ck < r\}}\geq   \ind_{\{ S_h^\ck <
    r\}}
  \quad\text{and thus}\quad
  \liminf_{r \rightarrow \infty } \frac{u(r-S_h^\ck)}{u(r)} \, \ind_{\{ S_h^\ck
   < r\}}
  \geq 1.
  \end{equation}
Let $G$ be a
  nonnegative measurable function on $\T_1$. By Fatou lemma, we deduce
  from~\eqref{eq:K=Sr+graft}, \eqref{eq:K=Tmass}  and~\eqref{eq:ineg-u/u}
  that:
 \[
  \liminf_{r\rightarrow \infty }
    \N^\psi_r\left[  \ind_{\{\bH(U^r)>h\}}\, 
      G( \wbr(\ctm, U_h^r)) \right]
    \geq \E^\psi\left[    G(\wbr( \ck, h))
  \right]. 
\]
Since $\N^\psi_r[\bH(U^r) > h]\leq 1$  and,  by the previous inequality
with $G=1$, 
$ \liminf_{r\rightarrow \infty } \N^\psi_r[\bH(U^r) > h]\geq 1$, we get
that:
\begin{equation}
   \label{eq:NrH>h}
  \lim_{r\rightarrow \infty } \N^\psi_r[\bH(U^r) > h]= 1.
\end{equation}
Now assume that $G$ is also bounded by $1$ and set $G^*=1-G$. We get:
\begin{align*}
   \limsup_{r\rightarrow \infty } 
\N^\psi_r\left[  \ind_{\{\bH(U^r)>h\}}\, 
  G( \wbr(\ctm, U_h^r)) \right]
  &= 1-\liminf_{r\rightarrow \infty }
    \N^\psi_r\left[  \ind_{\{\bH(U^r)>h\}}\, 
  G^*( \wbr(\ctm, U_h^r)) \right]\\
&\leq  1-  \E^\psi\left[    G^*(\wbr( \ck, h)) \right]\\
&=  \E^\psi\left[    G(\wbr( \ck, h)) \right].
 \end{align*}
This implies that for any  bounded nonnegative measurable function $G$ on
$\T_1$:
\[
 \lim_{r\rightarrow \infty } 
\N^\psi_r\left[  \ind_{\{\bH(U^r)>h\}}\, 
  G( \wbr(\ctm, U_h^r)) \right]  
=\E^\psi\left[    G(\wbr( \ck, h)) \right],
\]
and thus  for any $h\geq 0$ and any  bounded nonnegative measurable function $F$ on
$\T$ such that $F(\bt)=0$ on $\bH(\bt)<h$:
\[
 \lim_{r\rightarrow \infty } 
\N^\psi_r\left[  
 F( \br_h(\ctm)) \right]  
=\E^\psi\left[  F(\br_h( \ck)) \right]. 
\]
Then use Corollary~\ref{cor:mas} to get
$\N^\psi_r\left[  
 F( \br_h(\ctm)) \right]  =  \E^{\varphi, r}\left[  
 F( \br_h(\ckm)) \right]  $
so that~\eqref{eq:main-ctm} holds for $F$ such that
$F(\bt)\ind_{\{\bH(\bt)<h\}}=0$. 
To conclude use Corollary~\ref{cor:mas} and then~\eqref{eq:NrH>h} to get
that
$\lim_{r \rightarrow  \infty }  \E^{\varphi, r} [\bH(\ckm)  < h]=\lim_{r
  \rightarrow \infty } \N^\psi_r [\bH(\ctm) < h]=0 $.
\end{proof}

We get in particular the following (weaker) result.

\begin{theo}[Local convergence for the mass conditioning in the critical
  case]
\label{theo:cv-loc-mass-crit}
  Assume~\eqref{eq:hyp-param}-\eqref{eq:decroissant} hold (in particular the branching meachnism $\psi$  is
  critical, that is $\psi'(0)=0$).
The following 
local   convergence holds in distribution:
 \[
   \ctm  \xrightarrow[r\rightarrow +\infty]{\text{(d)}} \ck
     \quad\text{in}\quad
  (\T, \dlgh).
\]
\end{theo}

As it  shall be  used in a  forthcoming work, we  also mention  the next
result,  whose  proof   is  immediate  thanks  to~\eqref{eq:K=Sr+graft},
\eqref{eq:K=Tmass}  and~\eqref{eq:ineg-u/u}  (with   the  choice  of  the
regular  version  of  the   conditional  distribution  $\N^\psi_r$  from
Corollary~\ref{cor:mas}).

\begin{cor}[Lower bound  for the mass conditioning]
  \label{cor:mino-mass}
   Assume~\eqref{eq:hyp-param}-\eqref{eq:decroissant} hold  (in particular the branching meachnism $\psi$  is
  critical).  Let $r\in (0, +\infty )$, $\ctm$
  be the Lévy tree  under $\N^\psi$ conditioned to have total mass
  $r$ and   $U^r$  a leaf  of  $\ctm$  chosen uniformly.  For any nonnegative
  measurable functions $F$ and $G$  defined respectively on $\T$ and $\T_1$,  we have that  for all
  $h>0$:
 \begin{align*}
   \N^\psi_r\left[  \ind_{\{\bH(U^r)>h\}}\, 
   F( \br_h(\ctm)) \right]
&\, \geq \, 
 \E^\psi\left[ \ind_{\{ S_h^\ck < r\}}\,
    F(\br_h( \ck))
  \right],\\
     \N^\psi_r\left[  \ind_{\{\bH(U^r)>h\}}\, 
   G( \wbr_h(\ctm), U^r_h) \right]
&\, \geq \, 
 \E^\psi\left[ \ind_{\{ S_h^\ck < r\}}\,
    G(\wbr_h( \ck), h)
  \right].
  \end{align*} 
\end{cor}

\subsection{Convergence for the conditioning by the total
  mass in super-critical cases and some  sub-critical cases}
\label{sec:mass-super}

Intuitively, in the critical case $\alpha=0$, $u^{(0)}(r+a)/ u^{(0)}(r)$
``converges'' to 1 as  $r$ goes to infinity (this is  indeed the case in
the stable case); but this is  no more expected in the sub-critical case
(in    the     sub-critical    quadratic    case,    one     has    that
$\log(u^{(\alpha)}(r+a)/  u^{(\alpha)}(r))$ converges  to a  non trivial
constant times  $a$).  We shall  use the same  trick as in  the discrete
case  for  Bienaymé-Galton-Watson   (BGW)  tree,  see~\cite{Abraham2014}
or~\cite{Abraham2026} an the references therein, where one can exhibit a
parameterized   family  of   offspring  distributions   such  that   the
corresponding conditioned BGW  trees have the same  distribution.  To do
so,    we   consider    the    Girsanov    transformation   stated    in
Section~\ref{sec:meas-LT}.  Recall the set $\Theta^\psi$ defined therein.


Notice that  under~\eqref{eq:hyp-param}-\eqref{eq:variation} the
function $\psi$ is strictly convex, thus  there exists at most one root of $\psi'=0$  in $
\Theta^\psi$; such value when it exits will be  denoted $\theta^*$, and then the
branching mechanism $\psi_{\theta^*}$ and the density
$f_\sigma^{\theta^*}$ (which exists if and only if the distribution of
$\sigma$ on $(0, +\infty )$ under $ \N^\psi$ has a density, see~\eqref{eq:density-q}), simply denoted $\psi_*$ and $f_\sigma^*$, is
critical. 
In this section, we shall consider the following assumption:
\begin{equation}
   \label{eq:theta*}\tag{H7}
   \boxed{\text{There exists $\theta^*\in \Theta^\psi$ such that $\psi'(\theta^ *)=0$. 
}}
\end{equation}
Under~\eqref{eq:hyp-param}, Assumption~\eqref{eq:theta*} is  trivially satisfied in the critical case
(with $\theta^*=0$) and in the super-critical case (where $\psi'(0)<0$) under
the sufficient condition~\eqref{eq:variation}. In the sub-critical case, Assumption~\eqref{eq:theta*}
is  not satisfied in general. For example in the sub-critical case
$\psi(\lambda)=\alpha \lambda + \lambda^{a}$, with $\alpha>0$ and $a\in (1,2)$, then 
$\pi(\rd r) \propto r^{-a-1} \, \rd r$ so that  $\Theta^\psi=[0, \infty )$
and thus~\eqref{eq:theta*} is not satisfied. Notice~\eqref{eq:theta*} is satisfied for $a=2$ and
that $\Theta^\psi=\R$ in this case.

\medskip

Under~\eqref{eq:hyp-param},~\eqref{eq:variation} and~\eqref{eq:theta*}, let $\ck^*$ be the Kesten tree associated with the
critical branching
mechanism $\psi_*$.  The following result is a direct consequence of Theorem~\ref{theo:cv-loc-mass-crit} (with $\psi$ replaced by the
critical branching mechanism $\psi_*$)
and 
the
Girsanov transformation which implies, thanks
to~\eqref{eq:Girsanov-sigma},  that:
\[
  \N^\psi_r[\rd
  \ct]=\N^{\psi_\theta}_r[\rd \ct]
  \quad\text{for all $\theta\in \Theta^\psi$ and $r\in (0, +\infty )$.}
\]

\begin{cor}[Local convergence for the mass conditioning]
\label{cor:cv-loc-mass-generic}
Assume~\eqref{eq:hyp-param}-\eqref{eq:grey} and~\eqref{eq:theta*} hold
(recall that if $\psi$ is (super-)critical, then~\eqref{eq:theta*} holds), 
and the function
$r \mapsto r \expp{- \psi(\theta^*) r} f_\sigma(r)$ is non
increasing on $(0, +\infty )$ (that is, \eqref{eq:decroissant} holds with 
$f_\sigma$ replaced by  $f^*_\sigma$). 
The following 
local   convergence holds in distribution:
 \[
   \ctm  \xrightarrow[r\rightarrow +\infty]{\text{(d)}} \ck^*
     \quad\text{in}\quad
  (\T, \dlgh).
\]
\end{cor}

\begin{rem}[The non generic case]
  \label{rem:mass-condensation}
  The   so-called   non   generic   case   corresponds   to   the   case
  where~\eqref{eq:theta*} is not satisfied; this can only happen when
  $\psi$ is sub-critical.  In the discrete setting for
  BGW tree, see~\cite{js11,Janson2012} or~\cite{Abraham2014a}, the local
  limit of sub-critical BGW trees conditioned  to have a large number of
  vertices is  a tree with a  vertex at random finite  distance from the
  root  with infinitely  many  children.  We expect  also  to observe  a
  condensation  at finite  height  for  the local  limit  of Lévy  trees
  conditioned to  have a very  large mass when~\eqref{eq:theta*}  is not
  satisfied        (possibly        under        other        hypothesis
  than~\eqref{eq:hyp-param}-\eqref{eq:densite}): however the limit would
  not be locally compact and thus  the convergence would not be in $\T$,
  see also  the discussion  from Remark~\ref{rem:sub-critic-d}  where we
  could use therein a  trick to represent the ``local  limit'' with condensation
  at random finite height $E$.
\end{rem}

\subsection{Proof of Theorem~\ref{theo:mass}}
\label{sec:proof-theo-mass}
Let $\ct$ be a Lévy tree with distribution $\N^\psi[\rd \ct]$ and $U$ a
leaf of $\ct$ chosen according to the  mass measure $\bm$. Recall
that $\sigma$ is the total mass of the  measure $\bm$. 
Let  $G$ be a  nonnegative
  measurable function defined on the set $\T_1$  of pointed locally
  compact rooted trees.  By construction, we have with $U^r$
  as in Theorem~\ref{theo:mass}: 
\[
  \N^\psi\left[G(\ct, U)\right]=\int_0^\infty
  \N^\psi_r\left[G(\ctm, U^r)\right]\, 
  rf_\sigma(r)\, 
\rd r.
\]
We deduce that for $\gamma\geq 0$, $h>0$, $U^r_h$ defined in Theorem~\ref{theo:mass},
and  $U_h$ defined similarly as the ancestor of $U$ at level $h$ on the
event $\{\bH(U)> h\}$, that:
\begin{equation}
   \label{eq:A1}
  \N^\psi\left[\ind_{\{\bH(U)>h\}}\, \expp{-\gamma \bH(U)} \, G(\ct, U_h) \right]
    = \int_0^\infty   \N^\psi_r\left[\ind_{\{\bH(U^r)>h\}} \expp{-\gamma \bH(U^r)} \, G(\ctm, U^r_h)\right]\,
  rf_\sigma(r)\,  
\rd r.
\end{equation}

For a pointed tree $(T,d, \root,  x)\in \T_1$, where $\root, x\in
T$, we uniquely decompose the tree $T$ according to the branch $\II{\root, x}$
  and the sub-trees grafted on this branch:
  \[
    T= \II{\root, x}\circledast_{i\in I}(T_i,x_i),
    \]
    where $x_i$ belongs to the branch $\II{\root, x}$ and $T_i$ is the
      union of all the connected components of $T\setminus \{x_i\}$ not
    containing the root nor $x$, with $x_i$ added as a root. Notice this
    decomposition is unique (and measurable using an adaptation of~\cite[Proposition~5.32]{adhe-2022}).
    We shall consider a particular choice of function $G$ defined  by:
 \begin{equation}
   \label{eq:G=expg}
      G(T,x)=\exp{\left(  - \sum_{i\in I} G_0(d(\root, x_i), T_i)\right)},
    \end{equation} 
      with  $G_0$  a nonnegative measurable function defined on
      $\R_+\times\T $.

Set $\alpha=\psi'(0)\geq 0$ and for $\lambda,\gamma\geq 0$ and $h>0$:
\begin{align*}
  A(\gamma)
  &= \int_0^\infty \expp{-\lambda r}\,
     \N^\psi_r\left[\ind_{\{\bH(U^r)>h\}} \, \expp{-\gamma \bH(U^r)} \,
             G(\wbr(\ctm, U^r_h))\right]\, 
     rf_\sigma(r)\,  
     \rd r,\\
  B(\gamma)
  &=
  \int_{[0, +\infty )} \expp{-\lambda r} \pot^{(\alpha)} (\rd r)\,
   \E^{\varphi, r}\left[ \ind_{\{ \zeta^r>h\}}\, \expp{- \gamma \zeta^r }\, G(\wbr( \ckm, h)) 
   \right].
\end{align*}
To prove
that~\eqref{eq:K=Tmass} holds $\rd r$-a.e., it is enough in view
of~\eqref{eq:r*pi=u} to prove that 
$A(0)=B(0)$ for all $\lambda>0$, $h\geq 0$ and functions  $G$ given by~\eqref{eq:G=expg}  with
$G_0$ any  nonnegative measurable function such that $G_0(s, \cdot)=0$ for $s\geq h$.

Set:
\[
  F(h)=\alpha h+
\int_0^h\rd s\,  \bN^\psi[ 1-
\expp{- \lambda \sigma - G_0(s, \ct)} ].
\]
On the one hand, using~\eqref{eq:A1} and the decomposition of the Lévy tree along an
ancestral line which is a consequence of~\eqref{eq:Levy-Kh}
 at a given level $a$ and the
occupation formula~\eqref{eq:int-la}
for the  local times, we obtain:
\begin{align*}
  A(\gamma)
&= \N^\psi\left[ \expp{-\gamma \bH(U)} \,
\ind_{\{\bH(U)>h\}} \,
   \expp{-\lambda \sigma -\sum_{i\in I_U}
  G_0(h^U_i,
  \ct^U_i)  \ind_{\{h^U_i\leq  h\}}}
  \right]\\
   &= \int_h^\infty  \rd a\, \expp{-
   \gamma a }\,  \N^\psi\left[ \int_\ct \ell^a (\rd x)  \expp{-\sum_{i\in I}
   \left( \lambda \sigma^x_i + G_0(h^x_i,
    \ct^x_i)  \ind_{\{h^x_i\leq  h\}}\right)} \right]\\
  &=\int_h^\infty  \rd a\, \expp{ -\gamma a}\, \expp{ -\alpha a - \int_0^a \rd s \, \bN^\psi\left[1-
        \expp{-\lambda \sigma - G_0(s,\ct)\ind_{\{s\leq h\}}} \right]}\\
  &=\int_h^\infty  \rd a\, \expp{ -(\gamma+\alpha) a   - \int_0^h \rd s \, \bN^\psi\left[1-
        \expp{-\lambda \sigma - G_0(s,\ct)} \right]  - \int_h^a  \rd s \, \bN^\psi\left[1-
        \expp{-\lambda \sigma } \right]}\\
 &= \int_h^\infty  \rd a\, \expp{-
      \gamma a -   F(h) - (a-h) (\alpha+\varphi(\lambda))    
  }\\
  &= \frac{\expp{- \gamma h - F(h)}}{\gamma+ \psi'\circ \psi^{-1}  (\lambda)},
\end{align*}
where  for  the  fith  equality  we  used~\eqref{eq:bN-moment}  to  get
$\int_h^a  \rd s  \, \bN^\psi\left[1-  \expp{-\lambda \sigma  } \right]=
(h-a) \varphi(\lambda) $ with $\varphi= \psicr'\circ \psi^{-1}$ and that
$\alpha+\varphi= \psi' \circ \psi^{-1}$.

On the other hand, set
 $\bS_{[0, h]}=(S_t)_{t\in [0, h]}$ and 
 $Q_h= \E^{\varphi,r}\left[G(\wbr( \ckm, h))
\, \big|\, \bS_{[0, h]}\right]$. 
We obtain:
  \begin{align*}
    B(\gamma)
&=  \int_{(0, +\infty )} \expp{-\lambda r} \pot^{(\alpha)} (\rd r)\,
   \E^{\varphi,r}\left[ \ind_{\{ \zeta^r>h\}}\, \expp{- \gamma \zeta^r }\, Q_h
    \right]
    \\
&= \expp{-\gamma h }  \int_{(0, +\infty )} \expp{-\lambda r} \rd r\, 
    \int_0^\infty  \gamma  \expp{-\gamma s} \, \rd s\, u^{(\alpha)}(r)\,  \E^{\varphi,r}\left[ 
\left( \ind_{\{ \zeta^r>h\}}-  \ind_{\{ \zeta^r>s+h\}}\right)\, Q_h 
    \right]
    \\
&= \expp{-\gamma h }  \int_{(0, +\infty )} \expp{-\lambda r} \rd r\, 
  \int_0^\infty  \gamma  \expp{-\gamma s} \, \rd s\\
    & \hspace{2cm} \E^\varphi\left[\left(\expp{-\alpha h}\, 
   u^{(\alpha)}  (r- S_h)   \ind_{\{ S_h<r\}} -
 \expp{-\alpha (s+h)}\,  u^{(\alpha)} (r- S_{s+h}) 
 \ind_{\{ S_{h+s}<r\}}\right)\, Q_h 
    \right]
    \\
&= \frac{ \expp{-(\gamma+\alpha) h }}  {\psi'\circ \psi^{-1}(\lambda)}
    \int_0^\infty  \gamma  \expp{-\gamma s} \, \rd s\,\E^\varphi\left[
  \left(\expp{-\lambda S_h} - \expp{-\lambda S_{h+s} - \alpha s}\right)
\, Q_h 
    \right]
    \\
&=\frac{ \expp{-(\gamma+\alpha) h }}  {\psi'\circ \psi^{-1}(\lambda)}
    \int_0^\infty  \gamma \expp{-\gamma s} \, \rd s\,\E^\varphi\left[
  \expp{-\lambda S_h } \left(1- \expp{-s (\alpha+ \varphi(\lambda))} \right)
\, Q_h 
  \right]\\
&=
\frac{ \expp{-(\gamma+\alpha) h }}  {\gamma+\psi'\circ \psi^{-1}(\lambda)}
  \,\E^\varphi\left[
  \expp{-\lambda S_h }
\, Q_h 
  \right]  ,
  \end{align*}
where we used  $ \expp{- \gamma z}
\ind_{\{z>h\}} =\expp{-\gamma h} \int_0^\infty  \gamma \expp{-\gamma s}
(\ind_{\{z>h \}} - \ind_{\{z> h+s\}})\, \rd s$ and the density $u^{(\alpha)}$ of the
potential $\pot^{(\alpha)}$ for the second equality, \eqref{eq:cond-S-die} for the third as $Q_h$ is
$\cf_h$-measurable with $\P^{\varphi, r}=\P^{\varphi,r}_0$ (see also~\eqref{eq:K=Sr+graft}), the equalities:
\[
  \int_z^\infty  \rd r \expp{-\lambda r } u^{(\alpha)}(r-z)
= \expp{-\lambda z} \int_{[0, \infty )} \expp{-\lambda r} \, \pot^{(\alpha)}(\rd r)
=\frac{  \expp{-\lambda
  z}}{\alpha+\varphi(\lambda)}=\frac{  \expp{-\lambda
  z}}{\psi'\circ \psi^{-1}(\lambda)}
\]
for the fourth, and the Markov property of $\bS$ at
time $h$ for the fifth.
To conclude, we use the representation of the tree $ \ckm$, see~\eqref{eq:def-ckm},
to get:
\[
   \E^\varphi\left[
  \expp{-\lambda S_h} 
\, Q_h 
  \right]
=  \E^\psi\left[ \expp{- \sum_{h_i\leq  h} \lambda \sigma_i + G_0(h_i,
    \ct_i)}\right]
=  \expp{-\int_0^h \rd s\,  \bN^\psi\left[1 - \expp{-\lambda \sigma -
      G_0(s, \ct)}\right]}
=\expp{- F(h)+\alpha h},
\]
so we obtain:
\[
    B(\gamma)  =\frac{ \expp{-\gamma h -F(h)}  }{\gamma + \psi'\circ \psi^{-1}(\lambda)}\cdot
\]
We deduce that  $A(\gamma)=B(\gamma)$ for $\gamma\geq 0$ and thus $A(0)=B(0)$. So
the proof is complete.

\bibliographystyle{abbrv}
\bibliography{biblio}

\cuthere \medskip
\begin{center}
 \textbf{Index of notation}
\end{center}
\vspace{1em}

\renewcommand\labelitemi{-}
\setlength{\columnseprule}{1pt}
\begin{multicols}{2}
\noindent
\hrulefill

\begin{center}
   \textbf{Trees and pointed trees}
\end{center}

\begin{itemize}[leftmargin=*,itemsep=0.5em]
\item $\bt$, $T$, $\tau$, $\ct$, $\ck$: generic notations for  trees (or class of
  equiv.\ trees).
\item  $(\bt, x)$: a  (or a class of
  equiv.\ of)  tree with distinguished vertex $x\in \bt$.
\item $d$: generic distance on a tree.
\item  $\root$: generic notation for the root of a tree.

 \item $\bH(x)$ height of vertex $x$ (or distance from $x$ to $\root$). 
 
\item $R_h(\bt)$: the tree $\bt$ truncated at level $h$.

  \item $\wbr(\bt, x)$: the tree $\bt$ without the subtree above the
    vertex $x$.

\item $\bH(\bt)$: height of the tree $\bt$.
\item  $\lb x, y \rb$: the branch  joining the vertices $x$ to $y$.

\item $\II{0, h}$: the segment $[0,h]$ seen as a tree with root
  $\root=0$, and $h$ as a distinguished vertex. 
  
\item $\spine$: the infinite spine tree with root
  $\root=0$.

\item $\T$: Polish space  of (equiv.\ class of) rooted loc.\ compact closed trees.
\item $\T_1$: Polish space  of (equiv.\ class of) rooted loc.\ compact closed
  pointed trees.

  \end{itemize}
  \hrulefill

\begin{center}
   \textbf{Functions and random variables}
\end{center}

\begin{itemize}
\item $\psi$ branching mechanism with Lévy measure $\pi$ and quadratic
  parameter $\beta$.
\item $\bar \pi(r)=\pi((r, \infty ))$.
  \item $\bY$ CB starting at $x$ with branching mechanism  $\psi$ under
    $\P^\psi_x$.
  \item $\psi_\theta(\cdot)=\psi(\cdot + \theta) - \psi(\theta)$.

  \item $\psicr(\lambda)= \psi(\lambda) - \psi'(0) \lambda$.
    \item   Laplace exponent  $\varphi=\psicr'\circ \psi^{-1}$ from
      Section~\ref{sec:density-sigma}.
    \item $\pot^{(\alpha)}$: $\alpha$-potential of subordinator with
      Laplace exponent $\varphi$.
      \item $u^{(\alpha)}$ density of  $\pot^{(\alpha)}$. 
\end{itemize}

  \begin{center}
   \textbf{Lévy tree and Kesten tree}
 \end{center}

\begin{itemize}
\item $\ct$: a Lévy tree.

  \item $\ct_{[r]}$: a Lévy tree with root with ``size'' $r$.
 
\item $\ck$: a Kesten tree, with infinite spine $\spine$. 
\item $\wck_h= \wbr (\ck, h)$: the Kesten tree whose infinite spine is
  truncated at level $h$. 

 \item $\ell^a(\rd x)$: the local time at level $a$ on a Lévy
   tree.

 \item $\ell^{a,\ck}(\rd x)$: the local time  at level $a$ on the Kesten
   tree $\ck$.

\item $S^\ck_h$ total mass of the trees grafted on the spine of
  the Kesten tree up to level $h$.

 \item $\sigma=\sigma^\ct$ total mass of the Lévy tree $\ct$.
  
\item $f_\sigma$ ``density'' of $\sigma$ under $\N^\psi$.

 \item $\cth$: Lévy tree cond.\ to have height $h$.


\item   $ \ctn$: Lévy tree cond.\ to have  one vertex  of maximal
  size $\delta$.

  \item   $ \ctm$: Lévy tree cond.\ to have  total mass~$r$.

\end{itemize}

 \hrulefill
  \begin{center}
   \textbf{Probability and excursion measures}
 \end{center}

 \begin{itemize}
 \item $\N^\psi[\rd \ct]$: excursion measure for Lévy trees.
 \item $\P_r^\psi$: distribution of $\ct_{[r]}$.
\item $\bN^\psi$: intensity of the grafting on the infinite spine of the
  Kesten tree.
\item $\P^\psi$: distribution of the Kesten tree.

\end{itemize}

\end{multicols}
\end{document}